\documentclass[11pt]{article}

\usepackage[margin=1.1in]{geometry}  
\usepackage{graphicx}               
\usepackage{amsmath}              
\usepackage{amsfonts}              
\usepackage{bm}                        
\usepackage{amsthm}                
\usepackage{mathtools}             
\usepackage{esint}                     
\usepackage{color}
\usepackage{amssymb}             
\usepackage[hidelinks]{hyperref}
\usepackage[noabbrev]{cleveref}
\usepackage{authblk}                 
\usepackage{tikz}                       
\usepackage{pgfplots}                
\usetikzlibrary{patterns}              

\newtheorem{thm}{Theorem}[section]
\newtheorem{lem}[thm]{Lemma}
\newtheorem{prop}[thm]{Proposition}
\newtheorem{cor}[thm]{Corollary}

\newtheorem{rmk}[thm]{Remark}

\DeclareMathOperator{\loc}{loc}

\DeclareMathOperator{\dist}{dist}
\DeclareMathOperator{\per}{per}
\DeclareMathOperator{\cri}{cr}
\DeclareMathOperator{\supp}{supp}

\newcommand{\RR}{\mathbb{R}}     
\newcommand{\NN}{\mathbb{N}}     
  
\newcommand{\J}{\mathcal{J}}  
\newcommand{\K}{\mathcal{K}}
\newcommand{\R}{\mathcal{R}}
\newcommand{\LN}{\mathcal{L}^N}
\newcommand{\e}{\varepsilon} 
\newcommand{\ga}{\gamma}
\newcommand{\la}{\lambda}  
          
\numberwithin{equation}{section} 

\begin{document}

\nocite{}

\title{On the existence of non-flat profiles for a Bernoulli free boundary problem}

\author{
    Giovanni Gravina\\
    Department of Mathematical Analysis\\
    Faculty of Mathematics and Physics \\
    Charles University\\
    Prague, Czech Republic\\
    gravina@karlin.mff.cuni.cz
  \and
    Giovanni Leoni\\
    Department of Mathematical Sciences\\
    Carnegie Mellon University\\
    Pittsburgh, PA, USA\\
    giovanni@andrew.cmu.edu
    }
\date{\today}

\maketitle

\begin{abstract}
In this paper we consider a large class of Bernoulli-type free boundary problems with mixed periodic-Dirichlet boundary conditions. We show that solutions with non-flat profile can be found variationally as global minimizers of the classical Alt-Caffarelli energy functional. 
\vspace{.3cm}
\vspace{.2cm}
\end{abstract}

\section{Introduction}
In the classical paper \cite{MR618549}, Alt and Caffarelli used a variational approach to study the existence and regularity of solutions to the one-phase free boundary problem
\begin{equation}
\label{freebdrypb}
\left\{
\arraycolsep=1.4pt\def\arraystretch{1.6}
\begin{array}{rll}
\Delta u = & 0 & \text{ in } \Omega \cap \{u > 0\}, \\
u = & 0 & \text{ on } \Omega \cap \partial\{u > 0\}, \\
|\nabla u| = & Q & \text{ on } \Omega \cap \partial\{u > 0\}, \\
u = & u_0 & \text{ on } \Gamma.
\end{array}
\right.
\end{equation}
Here $\Omega$ is an open connected subset of $\RR^N$ with (locally) Lipschitz continuous boundary and $Q$ is a nonnegative measurable function. Solutions to (\ref{freebdrypb}) are critical points for the functional
\begin{equation}
\label{J}
\J(u) \coloneqq \int_{\Omega}\big(|\nabla u|^2 + \chi_{\{u > 0\}}Q^2\big)\,d\bm{x}, \quad u \in \K,
\end{equation}
where 
\begin{equation}
\label{K}
\K \coloneqq \{u \in L^1_{\loc}(\Omega) : \nabla u \in L^2(\Omega;\RR^N) \text{ and } u = u_0 \text{ on } \Gamma\},
\end{equation}
with $\Gamma \subset \partial \Omega$ a measurable set with $\mathcal{H}^{N-1}(\Gamma) > 0$ and $u_0 \in H^1_{\loc}(\Omega)$ a nonnegative function satisfying
\begin{equation}
\label{ju0finite}
\J(u_0) < \infty.
\end{equation}
The equality $u = u_0$ on $\Gamma$ is in the sense of traces. 

Under the assumption that $Q$ is a H\"older continuous function satisfying
\begin{equation}
\label{Qmin>0}
0 < Q_{\min} \le Q(\bm{x}) \le Q_{\max} < \infty,
\end{equation}
Alt and Caffarelli \cite{MR618549} proved local Lipschitz regularity of local minima and showed that the free boundary $\partial \{u > 0\}$ is a $C^{1,\alpha}$ regular curve locally in $\Omega$ if $N = 2$, while if $N \ge 3$ they proved that the reduced free boundary $\partial^{\operatorname{red}}\{ u > 0\}$ is a hypersurface of class $C^{1,\alpha}$ locally in $\Omega$, for some $0 < \alpha < 1$, and that the singular set
\[
\Sigma_{\operatorname{sing}} \coloneqq \Omega \cap \left\{\partial \{ u > 0\}  \setminus \partial^{\operatorname{red}}\{ u > 0\}\right\}
\] 
has zero $\mathcal{H}^{N - 1}$ measure. See also \cite{MR752578} for the quasi-linear case and \cite{MR2133664} for the case of the $p$-Laplace operator.

While the regularity of minimizers is optimal, the regularity of the free boundary for $N \ge 3$ was improved by Weiss in \cite{MR1759450}. Weiss, following an approach closely related to the theory of minimal surfaces and by means of a monotonicity formula, proved the existence of a maximal dimension $k^* \ge 3$ such that for $N < k^*$ the free boundary is a hypersurface of class $C^{1,\alpha}$ locally in $\Omega$, for $N = k^*$ the singular set $\Sigma_{\operatorname{sing}}$ consists at most of isolated points, and if $N > k^*$ then $\mathcal{H}^s(\Sigma_{\operatorname{sing}}) = 0$ for every $s > N - k^*$. In \cite{MR2082392}, Caffarelli, Jerison and Kenig proved the full regularity of the free boundary in dimension $N = 3$, thus showing that $k^* \ge 4$. They also conjectured that $k^* \ge 7$. In a later work, De Silva and Jerison exhibited an example of a global energy minimizer with non-smooth free boundary in dimension 7 (see \cite{MR2572253}); their result gives the upper bound $k^* \le 7$. More recently, Jerison and Savin showed that the only stable homogeneous solutions in dimension $N \le 4$ are hyperplanes, a result which implies full regularity of the free boundary for $N \le 4$, and consequently that $k^* \in \{5,6,7\}$ (see \cite{MR3385632}). We refer to the recent paper of Edelen and Engelstein (see \cite{MR3894044}) for more details on the structure of the singular set $\Sigma_{\operatorname{sing}}$.

As already remarked in \cite{MR618549}, if $N  = 3$ the energy functional $\J$ admits a critical point with a point singularity in the free boundary. Similar results have been obtained for two-phase free boundary problems (see \cite{MR732100}, \cite{caff1}, \cite{caff2}, \cite{caff3}). 

It is important to observe that the regularity of the free boundary is strongly related to the assumption $0 < Q_{\min} \le Q(\bm{x})$ in (\ref{Qmin>0}). Indeed, in the recent paper \cite{MR2915865} Arama and the second author showed that for $N = 2$ and in the special case in which 
\begin{equation}
\label{Qh}
Q(x,y) = \sqrt{(h - y)_+} \quad \text{ for some } h > 0,
\end{equation}
if a local minimizer $u$ has support below the line $\{y = h\}$ and if there exists a point $\bm{x}_0 = (x_0,h) \in \partial \{u > 0\} \cap \Omega$, then 
\begin{equation}
\label{Qdecay1/2}
|\nabla u(x,y)| \le C(h - y)^{1/2} \quad \text{ for } \bm{x} \in B_r(\bm{x}_0),
\end{equation}
provided $r$ is sufficiently small (see Remark 3.5 in \cite{MR2915865}), and, if in addition $u$ coincides with its symmetric decreasing rearrangement with respect to the variable $x$,
\[
u(0,y) \ge c(h - y)^{3/2} \quad \text{ for } y \in [0,h]
\] 
(see Theorem 5.11 in \cite{MR2915865}). On the other hand, using a monotonicity formula and a blow-up method, Varvaruca and Weiss (see Theorem A in \cite{MR2810856}) proved that for a suitable definition of solution if the constant $C$ in (\ref{Qdecay1/2}) is one then the rescaled function
\begin{equation}
\label{bund}
\frac{u(\bm{x}_0 + r\bm{x})}{r^{3/2}} \to \frac{\sqrt{2}}{3}\rho^{3/2}\cos\left(\frac{3}{2}\left(\min\left\{\max\left\{\theta,-\frac{5\pi}{6}\right\},-\frac{\pi}{6}\right\} + \frac{\pi}{2}\right)\right) \quad \text{ as } r \to 0^+,
\end{equation}
strongly in $W^{1,2}_{\loc}(\RR^2)$ and locally uniformly on $\RR^2$, where $(x,y) = (\rho \cos \theta, \rho \sin \theta)$, and near $\bm{x}_0$ the free boundary $\partial \{u > 0\}$ is the union of two $C^1$ graphs with right and left tangents at $\bm{x}_0$ forming an angle of $2\pi/3$ (see also \cite{MR2928132}). This type of singular solutions are related to Stokes' conjecture on the existence of extreme water waves (see \cite{stokes}). Indeed, when $N = 2$, $Q$ takes the form $(\ref{Qh})$,
\begin{equation}
\label{defww}
\Omega \coloneqq (-\la/2,\la/2) \times (0, \infty),\ \quad \Gamma \coloneqq (-\la/2,\la/2) \times \{0\},\ \quad u_0 \equiv m,
\end{equation}
the free boundary problem (\ref{freebdrypb}) describes gravity waves of permanent form on the free surface of an ideal fluid. The motion is assumed to be irrotational and two dimensional (see \cite{MR0112435}).

The existence of extreme waves and the corner singularity have been proved in a series of papers (see \cite{MR869412}, \cite{MR666110}, \cite{MR1446239}, \cite{MR1883094}, \cite{MR513927}; see also \cite{MR2604871}, \cite{MR0502787}, \cite{MR869413}, \cite{MR2038344} for the existence of regular waves) using a hodograph transformation to map the set $\{u > 0\}$ onto an annulus. 

The main drawback in proving the existence of regular and extreme water waves using the variational setting of $(\ref{J})$ is that \emph{global minimizers of the energy functional $\J$ specialized to the case $(\ref{Qh})$, $(\ref{defww})$ are one dimensional functions of the form $u = u(y)$}, which correspond to flat profiles (see Theorem 5.1 in \cite{MR2915865}). For this reason the paper \cite{MR2915865} gives interesting results only for local minimizers or when the Dirichlet boundary datum $u_0$ is not constant on the bottom. Necessary and sufficient minimality conditions in terms of the second variation of $\J$ have been derived by Fonseca, Mora and the second author in \cite{FLM}. We refer to the papers \cite{MR2027299}, \cite{MR2264220}, \cite{MR3689941}, \cite{MR3494328}, \cite{CHEN2018517}, \cite{MR2278405}, \cite{MR3811234}, \cite{MR3177719} and the references therein for alternative approaches to water waves.

The purpose of this paper is to show that by adding an additional Dirichlet boundary condition on part of the later boundary it is possible to construct global minimizers of $\J$ in the setting (\ref{Qh}), (\ref{defww}), which are not one dimensional. To be precise, we let $\Omega$ be the half-infinite rectangular parallelepiped 
\begin{equation}
\label{strip}
\Omega \coloneqq \R \times (0,\infty),
\end{equation}
where $\R$ is the open cube of $\RR^{N-1}$ with center at the origin and side-length $\la > 0$, that is,
\[
\R \coloneqq \left(-\frac{\la}{2} , \frac{\la}{2} \right)^{N-1}.
\]
We will impose periodic boundary conditions on the lateral portion of the boundary, therefore we will require that the class of admissible functions is a subset of the Sobolev space
\begin{multline}
\label{H1Rloc}
H^1_{\per}(\Omega) \coloneqq \{u \in H^1_{\loc}(\RR^N_+) : u(\bm{x} + \la \bm{e}_i) = u(\bm{x}) \text{ for } \LN \text{-a.e.\@ } \bm{x} \in \RR^N_+ \\ \text{ and every } i = 1,\dots,N-1\}.
\end{multline}
With the choice 
\begin{equation}
\label{Qb}
Q(\bm{x}) \coloneqq (h - x_N)_+^b,
\end{equation}
where $b,h> 0$, the functional $\J$ in (\ref{J}) can be rewritten as
\begin{equation}
\label{Jper}
\J_h(u) \coloneqq \int_{\Omega}\big(|\nabla u|^2 + \chi_{\{u > 0\}}(h - x_N)_+^{2b}\big)\,d\bm{x}, \quad \text{ for } u \in \K_{\ga},
\end{equation}
where
\begin{equation}
\label{Kg}
\K_{\ga} \coloneqq \{u \in H^1_{\per}(\Omega) : u = u_0 \text{ on } \Gamma_{\ga}\},\ \quad \ga > 0.
\end{equation}
Here the Dirichlet datum $u_0$, defined by
\begin{equation}
\label{u_0gam}
u_0(\bm{x}) \coloneqq m\left(1 - \frac{x_N}{\ga}\right)_+,\ \quad m > 0,
\end{equation}
is prescribed on 
\begin{equation}
\label{gammaga}
\Gamma_{\ga} \coloneqq (\R \times \{0\}) \cup (\partial \R \times (\ga, \infty)).
\end{equation}
In particular, notice that $u_0$ is constant on $\R \times \{0\}$ and zero on $\partial \R \times (\ga,\infty)$.

One of our main results consists of proving that there are choices of the parameter $\ga$ (depending on $b,m$, and $h$, but independent of $\la$) which have the effect of eliminating trivial solutions from the domain of $\J_h$. This is made precise in the following theorem.

         
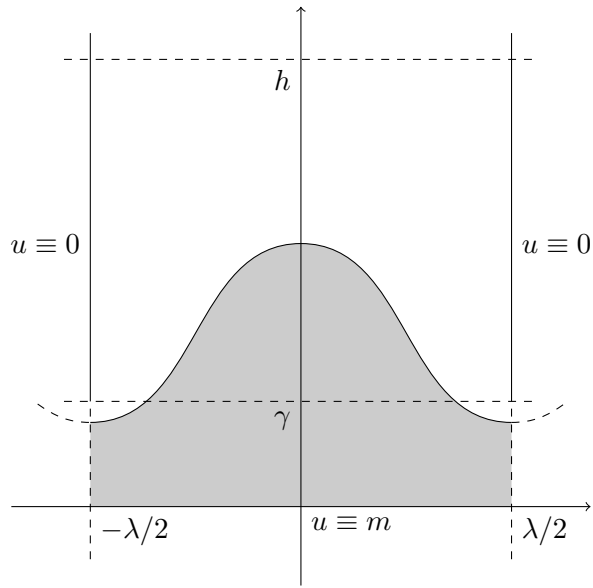
\begin{center}
         \begin{tikzpicture}[blend group=normal, scale=0.7]
         
         \draw [->] (5,-0.5) -- (5,10.5); 
         \draw [->] (-0.5,1) -- (10.5,1); 
         \draw [dashed] (1,0) -- (1,3); 
         \draw (1,3) -- (1,10); 
         \draw [dashed] (9,0) -- (9,3); 
         \draw (9,3) -- (9,10); 
         
         
         \draw [thin] (1,2.6) to [out=0,in=180] (5,6) to [out=0,in=180] (9,2.6);
         \draw [dashed] (0,2.95) to [out=-35,in=180] (1,2.6);
         \draw [dashed] (9,2.6) to [out=0,in=215] (10,2.95);
         \draw [thin, fill=black, fill opacity=0.2, draw opacity=0] (1,1) -- (1,2.6) to [out=0,in=180] (5,6) to [out=0,in=180] (9,2.6) -- (9,1);
         
         \draw [dashed] (0.5,3) -- (9.5,3);
         \draw [dashed] (0.5,9.5) -- (9.5,9.5);
         
         \node [below right] at (1,1) {$-\la/2$};
         \node [below right] at (9,1) {$\la/2$};
         \node [below right] at (5,1) {$u \equiv m$};
         \node [below left] at (5,3) {$\ga$};
         \node [below left] at (5,9.5) {$h$};
         \node [left] at (1,6) {$u \equiv 0$};
         \node [right] at (9,6) {$u \equiv 0$};
         \end{tikzpicture}
         \end{center}
         
\begin{thm}[Existence of non-flat minimizers]
\label{exofgamma}
Given $b, m, h, \la > 0$, let $\Omega$, $\J_h$, and $\K_{\ga}$ be defined as in $(\ref{strip})$, $(\ref{Jper})$, and $(\ref{Kg})$, respectively. Let 
\begin{equation}
\label{h*}
h^\# \coloneqq \frac{b + 1}{b^{b/(b + 1)}} m^{1/(b + 1)}, \quad h^* \coloneqq \frac{2b + 2}{(2b+1)^{b/(b + 1)}} m^{1/(b + 1)},
\end{equation}
and, for $h > h^{\#}$, let $t_h$ be the first positive root of the polynomial 
\[
p(t) \coloneqq t^2(h - t)^{2b} - m^2.
\]
Furthermore, for $h \in (h^{\#},h^*)$, let $\tau_h > t_h$ be the unique value such that
\[
\frac{m^2}{t_h} + \frac{h^{2b + 1} - (h - t_h)^{2b + 1}}{2b + 1} = \frac{m^2}{\tau_h} + \frac{h^{2b + 1} - (h - \min\{h,\tau_h\})^{2b + 1}}{2b + 1},
\]
and set $\tau_h = t_h = h/(b + 1)$ if $h = h^{\#}$. Then every global minimizer $u \in \K_{\ga}$ of the functional $\J_h$ is not of the form $u = u(x_N)$ provided
\begin{equation}
\label{gammarange}
\left\{
\arraycolsep=1.4pt\def\arraystretch{1.6}
\begin{array}{rll}
\ga \in & (0,\infty) & \text{ if } h < h^{\#}, \\
\ga \in & (0,t_h) \cup (\tau_h,\infty) & \text{ if } h^{\#} \le h < h^*, \\
\ga \in & (0,t_h) & \text{ if } h \ge h^*.
\end{array}
\right.
\end{equation}
\end{thm}

\begin{rmk}
The numbers $h^{\#}, h^*, t_h$, and $\tau_h$ arise naturally from the study of the minimization problem for a one-dimensional version of $\J_h$. The analysis of this auxiliary problem is presented in \Cref{auxsec}. In particular, in \Cref{gaequiv} we give an equivalent characterization of the different ranges in $(\ref{gammarange})$.
\end{rmk}

We then proceed to study qualitative properties of global minimizers as we vary the parameter $h$. One of the main results in this direction is an analogue to Theorem 5.6 in \cite{MR2915865}, which roughly speaking gives a characterization of the values of $h$ for which the support of global minimizers stays bounded. The key ingredients of our proof are the monotonicity techniques developed in Section 5 of \cite{MR647374}, Theorem 10.1 in \cite{MR1009785}, and ideas borrowed from the proof of the continuous fit as presented in Section 9 of \cite{MR733897}. 

\begin{thm}[Existence of a critical height]
\label{exofhcrit}
Given $b, m, \la > 0$, let $\theta \colon \RR_+ \to \RR_+$ be a non-increasing function, set
\begin{equation}
\label{choiceg}
\gamma_h \coloneqq \theta(h),
\end{equation}
and for every $h > 0$ consider $\Omega$, $\J_h$, and $\K_{\ga_h}$ defined as in $(\ref{strip})$, $(\ref{Jper})$, and $(\ref{Kg})$, respectively. Then there exists a threshold value for the parameter $h$, denoted by $h_{\cri}$, with the following properties:
\begin{itemize}
\item[$(i)$] $0 < h_{\cri} < \infty$;
\item[$(ii)$] for every $h > h_{\cri}$ and for every global minimizer $u \in \K_{\ga_h}$ of $\J_h$ the support of $u$ stays strictly below the hyperplane $\{x_N = h\}$;
\item[$(iii)$] for every $0 < h < h_{\cri}$ and for every global minimizer $u \in \K_{\ga_h}$ of $\J_h$ the support of $u$ crosses the hyperplane $\{x_N = h\}$ and therefore $u$ is positive in $\R \times (h,\infty)$.
\end{itemize}
\end{thm}

\begin{rmk}
Although \Cref{exofhcrit} holds for any choice of the non-increasing function $\theta$, it is of particular interest in the case in which for every $h > 0$ the value $\ga_h = \theta(h)$ satisfies $(\ref{gammarange})$.
\end{rmk}

Next, we give bounds on the critical height $h_{\cri}$ in terms of the Dirichlet datum $m$ and obtain in return a characterization of its asymptotic behavior.

\begin{thm}[Scaling of the critical height]
\label{scalingthm}
Under the assumptions of \Cref{exofhcrit}, if in addition $\ga_{t_{h^*}} \ge t_{h^*}$, we have
\[
h_{\cri} \sim m^{1/(b + 1)}.
\]
Here $t_{h^*}$ and $\ga_{t_{h^*}}$ are the numbers given in \Cref{exofgamma} and in $(\ref{choiceg})$ corresponding to $h = h^*$, where $h^*$ is defined in $(\ref{h*})$.
\end{thm}

Further properties of solutions to the minimization problem for $\J_h$ are summarized in the following theorem. 

\begin{thm}[Structure theorem]
\label{structure}
Under the assumptions of \Cref{exofhcrit}, if in addition $\theta$ is continuous, for every $h > 0$ there exist two (possibly equal) global minimizers of $\J_h$ in $\K_{\ga_h}$, namely $u_h^+,u_h^-$, with the following properties:
\begin{itemize}
\item[$(i)$] for any sequence $\{(h_n, u_n)\}_n$ such that $h_n \nearrow h$ and $u_n \in \K_{\ga_{h_n}}$ is a global minimizer of $\J_{h_n}$ we have that $\nabla u_n \to \nabla u_h^+$ in $L^2(\Omega;\RR^N)$, and $u_n \to u_h^+$ in $H^1_{\loc}(\Omega)$ and uniformly on compact subsets of $\Omega$;
\item[$(ii)$] for any sequence $\{(h_n, u_n)\}_n$ such that $h_n \searrow h$ and $u_n \in \K_{\ga_{h_n}}$ is a global minimizer of $\J_{h_n}$ we have that $\nabla u_n \to \nabla u_h^-$ in $L^2(\Omega;\RR^N)$, and $u_n \to u_h^-$ in $H^1_{\loc}(\Omega)$ and uniformly on compact subsets of $\Omega$;
\item[$(iii)$] if $w \in \K_{\ga_h}$ is a global minimizer of $\J_h$ then $u_h^- \le w \le u_h^+$;
\item[$(iv)$] $u_h^+,u_h^-$ are symmetric with respect to the coordinate hyperplanes $\{x_i = 0\}$, $i = 1, \dots, N - 1$ and coincide with their respective symmetric decreasing rearrangements with respect to the variables $x_1, \dots, x_{N - 1}$. 
\end{itemize}
Furthermore, the minimization problem for $\J_h$ in $\K_{\ga_h}$ admits a unique solution for all but countably many values of $h$.
\end{thm}

Finally, we remark that while the additional Dirichlet constraint $u = 0$ on $\partial \R \times (\ga_h,\infty)$ allows us to construct non-flat global minimizers, it has the disadvantage of potentially destroying the regularity of minimizers and their free boundaries at the interface $\partial \R \times \{\ga_h\}$, where one has Dirichlet boundary conditions on $\partial \R \times (\ga_h,\infty)$ and periodic boundary conditions on $\partial \R \times (0,\ga_h)$. 
         \begin{center}
         \begin{tikzpicture}[blend group=normal, scale=0.7]
         
         \draw [dashed] (2,0) -- (2,3); 
         \draw (2,3) -- (2,5.5); 
         
         \draw [dashed] (13,0) -- (13,1); 
         \draw (13,1) -- (13,5.5); 
         
         \draw [thin] (2,2) to [out=0, in=190] (8,5);
         \draw [dashed] (0, 2.583) to [out=-29.5, in=180] (2,2);
         \draw [thin, fill=black, fill opacity=0.2, draw opacity=0] (2,0) -- (2,2) to [out=0, in=190] (8,5) -- (8,0) -- (2,0);
         \draw [dashed] (2,2) -- (4.5,2);
         \draw [ultra thin] (2,2.3) -- (2.3,2.3) -- (2.3,2);
         
         \draw [thin] (13, 2) to [out=90, in=190] (19,5); 
         \draw [dashed] (11,4.305) to [out=-15.25, in=90] (13,2);
         \draw [thin, fill=black, fill opacity=0.2, draw opacity=0] (13,0) -- (13,2) to [out=90, in=190] (19,5) -- (19,0) -- (13,0);
         \draw [dashed] (13,2) -- (15.5,2);
         \draw [ultra thin] (13,2.3) -- (13.3,2.3) -- (13.3,2);
         
         \draw[fill] (2,3) circle [radius=0.06];
         \node [left] at (2,3) {$(-\la/2,\ga_h)$};
         \draw[fill] (13,1) circle [radius=0.06];
         \node [left] at (13,1) {$(-\la/2,\ga_h)$};
         \node [right] at (2,5) {$u \equiv 0$};     
         \node [right] at (13,5) {$u \equiv 0$};       
         \node [right] at (3,-1) {Figure A};
         \node [right] at (14,-1) {Figure B};
         
         \end{tikzpicture}
         \end{center}
Notice that due to the periodic boundary conditions below the line $\{y = \ga_h\}$, if the free boundary $\partial \{u > 0\}$ of a global minimizer $u \in \K_{\ga_h}$ of the functional $\J_h$ touches the fixed boundary strictly below the line $\{y = \ga_h\}$ (as in Figure A) then it must be regular across periods as a consequence of the interior regularity (see \Cref{exofmin}). In particular, in dimension $N = 2$, this implies that the free boundary hits the fixed boundary with a horizontal tangent and furthermore every global minimizer is a solution to (\ref{freebdrypb}) in the entire half-plane. On the other hand, if the free boundary $\partial \{u > 0\}$ of a global minimizer $u \in \K_{\ga_h}$ of the functional $\J_h$ touches the fixed boundary strictly above the line $\{y = \ga_h\}$ then we are in a position to apply the recent work of Chang-Lara and Savin \cite{MR3916702} (see also \cite{MR682265}, \cite{MR818805}, \cite{MR2047400}) in which it is shown that the free boundary of a viscosity solution of (\ref{freebdrypb}) detaches tangentially from a portion of the fixed boundary where $u$ vanishes and is a $C^{1,1/2}$ regular hypersurface locally in a neighborhood of $\partial \Omega$ (see Figure B). The result is obtained relating the behavior of the free boundary to a Signorini-type obstacle problem. In the remaining case for the two dimensional problem, i.e., when $(-\la/2,\ga_h)$ is an accumulation point for the free boundary, it was proved by the authors (see \cite{GL19} and \cite{mythesis}) that the free boundary of a minimizers which coincides with its symmetric decreasing rearrangement with respect to the variable $x$ must hit the fixed boundary with horizontal tangent (see Figure C).
         \begin{center}
         \begin{tikzpicture}[blend group=normal, scale=0.7]
         
         \draw [dashed] (2,0) -- (2,3); 
         \draw (2,3) -- (2,5.5); 
         
         \draw (2,3) to [out=0, in=190] (8,5);
         \draw [dashed] (0, 3.583) to [out=-29.5, in=180] (2,3);
         \draw [fill=black, thin, fill opacity=0.2, draw opacity=0] (2,0) -- (2,3) to [out=0, in=190] (8,5) -- (8,0) -- (2,0);
         \draw [dashed] (2,3) -- (4.5,3);
         \draw [ultra thin] (2,3.3) -- (2.3,3.3) -- (2.3,3);
         
         \draw[fill] (2,3) circle [radius=0.06];
         \node [left] at (2,2.5) {$(-\la/2,\ga_h)$};   
         \node [right] at (2,5) {$u \equiv 0$};      
         \node [right] at (3,-1) {Figure C};
         \end{tikzpicture}
         \end{center}
         
\begin{thm}
\label{mainthm}
Given $m, \la, h > 0$ and $\ga < h$, let $N = 2$, $b = 1/2$, and let $\Omega$, $\J_h$, and $\K_{\ga}$ be defined as in $(\ref{strip})$, $(\ref{Jper})$ and $(\ref{Kg})$, respectively. Let $u \in \K_{\ga}$ be a global minimizer of $\J_h$ which coincides with symmetric decreasing rearrangement with respect to the variable $x$ and assume that $\bm{x}_0 = (-\la/2,\ga)$ is an accumulation point for the free boundary on $\partial \Omega$, i.e.,
\[
\bm{x}_0 \in \overline{\partial\{u > 0\} \cap \Omega}.
\]
Then the portion of the free boundary $\partial \{u > 0\}$ in $\{\bm{x} \in \Omega : -\la/2 < x < 0\}$ can be described by the graph of a function $x = g(y)$ and furthermore, the free boundary meets the fixed boundary at the point $\bm{x}_0$ with horizontal tangent, i.e.,
\[
\lim_{y \to \ga}\frac{|g(y) - g(\ga)|}{|y - \ga|} = \infty.
\]
\end{thm}    
         
In conclusion, we would like to remark that Theorems \ref{exofgamma} - \ref{structure} are a preliminary step towards a variational proof of the existence of regular waves and of Stokes waves. Indeed, if one could show that for some particular choice of the parameters $m, \la, h, \gamma_h$ the free boundary touches the fixed boundary below or at the point $y = \ga_h$, then (see figures A and C) \Cref{exofhcrit} and \Cref{mainthm} would give a variational proof of the existence of regular waves established by Krasovski\u{\i} \cite{MR138284} and Keady and Norbury \cite{MR0502787}. In turn, if in this range of parameters we could show that the free boundary of $u_h$ approaches $\{x_N = h_{\cri}\}$ as $h \searrow h_{\cri}$ (see \Cref{exofhcrit}) this would give a variational proof of the existence of Stokes waves. Both problems are under study.

Independently of their applications to water waves, we believe that the techniques developed in this paper are of interest in themselves and could be applied to other free boundary problems.

Our paper is organized as follows: for the convenience of the reader, in \Cref{bgr} we recall some well-known results on the existence and regularity of minimizers of the energy functional $\J_h$. In \Cref{auxsec} we study an auxiliary one dimensional variational problem; the results of this section will be instrumental in \Cref{nontriv}, where we present the proof of \Cref{exofgamma}. \Cref{propertiesglobmin} is dedicated to the study of qualitative and structural properties of global minimizers. In particular, \Cref{propertiesglobmin} contains the proofs of \Cref{exofhcrit}, \Cref{scalingthm}, and \Cref{structure}.

\section{Background results}
\label{bgr}
In this section we collect well-known results concerning existence and regularity properties of solutions to the minimization problem for $\J_h$ in $\K_{\ga}$.

\begin{thm}
\label{exofmin}
Given $b, m, h, \ga, \la > 0$, let $\Omega$, $\K_{\ga}$, and $\J_h$ be defined as in $(\ref{strip})$, $(\ref{Jper})$, and $(\ref{Kg})$, respectively. Then the minimization problem for $\J_h$ in $\K_{\ga}$ admits a solution. Furthermore, if $u \in \K_{\ga}$ is a global minimizer of the functional $\J_h$, the following hold: 
\begin{itemize}
\item[$(i)$] $u$ is subharmonic in $\Omega$;
\item[$(ii)$] $u$ is locally Lipschitz continuous in $\Omega$; 
\item[$(iii)$] $u$ is harmonic in the set $\{u > 0\}$;
\item[$(iv)$] $u$ satisfies the free boundary condition $(\ref{freebdrypb})_3$ in a weak sense, i.e.,
\[
\lim_{\e \to 0^+}\int_{\partial \{ u > \e\}}\left(|\nabla u|^2 - (h - x_N)^{2b}_+\right)\eta \cdot \nu\,d\mathcal{H}^{N - 1} = 0 \quad \text{ for every } \eta \in C^{\infty}_c(\Omega;\RR^N);
\]
\item[$(v)$] for any $K$ compactly contained in $\R \times (0,h)$, the free boundary $\partial\{u > 0\} \cap K$ is a smooth hypersurface except possibly on a closed singular set $\Sigma_{\operatorname{sing}} \subset \partial \{u > 0\}$ of Hausdorff dimension $N - 5$, and 
\[
\partial_{-\nu}u(\bm{x}) = (h - x_N)^b, \quad \bm{x} = (\bm{x}',x_N) \in \partial\{u > 0\} \cap K \setminus \Sigma_{\operatorname{sing}}.
\] 
\end{itemize}
\end{thm}

\begin{proof}
Since $\J_h(u_0) < \infty$ for $u_0$ defined as in (\ref{u_0gam}), the proof of existence is essentially analogous to that of Theorem 1.3 in \cite{MR618549} (see also Theorem 2.2 in \cite{MR2915865}) and therefore we omit it. The proofs for statements $(i)$ through $(iv)$ can be found in \cite{MR618549}; more precisely, we refer to Lemma 2.2, Corollary 3.3, Lemma 2.4, and Theorem 2.5. Statement $(v)$ is  Corollary 1.2 in \cite{MR3385632}.
\end{proof}

\begin{rmk}
In view of property $(i)$ we can work with the precise representative 
\[
u(\bm{x}) = \lim_{r \to 0^+}\fint_{B_r(\bm{x})}u(\bm{y})\,d\bm{y}, \quad \bm{x} \in \Omega.
\]
\end{rmk}

Typically, a first step for the study of minimizers and their free boundaries is to obtain non-degeneracy estimates. The next proposition, reported below for future reference, is a classical result in this direction and is essentially due to Alt and Caffarelli (see Lemma 3.4 and Remark 3.5 in \cite{MR618549}; see also Theorem 3.6 and Remark 5.2 in \cite{MR2915865}). For the convenience of the reader, we adapt the statement to our framework.
\begin{prop}
\label{competition2}
Given $b,m,h,\ga, \la > 0$, let $\Omega, \J_h$, and $\K_{\ga}$ be defined as in $(\ref{strip})$, $(\ref{Jper})$, and $(\ref{Kg})$, respectively. Then for every $k \in (0,1)$ there exists a positive constant $C = C(N,k)$ such that for every minimizer $u$ of $\J_h$ in $\K_{\ga}$ and for every ball $B_r(\bm{x}) \subset \Omega$, if 
\[
\frac{1}{r} \fint_{\partial B_r(\bm{x})}u\, d\mathcal{H}^{N-1} \le C(h - x_N - kr)_+^b,
\]
then $u \equiv 0$ in $B_{kr}(\bm{x})$. Moreover, the result is still valid for balls not entirely contained in $\Omega$ if $u$ vanishes on $B_r(\bm{x}) \cap \partial \Omega$. In particular, this holds if $B_r(\bm{x}) \cap \partial \Omega \subset \partial \R \times (\ga,\infty)$. 
\end{prop}

\section{An auxiliary one-dimensional variational problem}
\label{auxsec}
This section is dedicated to the study of the minimization problem for the functional 
\begin{equation}
\label{1d}
\mathcal{I}_h(v) \coloneqq \int_0^{\infty}\big(v'(t) + \chi_{\{v > 0\}}(t)(h - t)^{2b}_+\big)\,dt,
\end{equation}
defined in the class
\begin{equation}
\label{K1d}
\K_{\ga,1\text{-d}} \coloneqq \{v \in L^1_{\loc}((0,\infty)) : v \in H^1((0,r)) \text{ for every } r > 0, v(0) = m, \text{ and } v(\ga) = 0\}.
\end{equation}
Our motivation for considering this problem comes from the following observation: if $u \in \K_{\ga}$ is of the form $u = u(x_N)$, then $u(0) = m$, $u(\ga) = 0$, and by Tonelli's theorem 
\begin{equation}
\label{tonelliJ}
\J_h(u) = \int_{\R}\int_0^{\infty}\big(|u'(x_N)|^2 + \chi_{\{u > 0\}}(\bm{x}',x_N)(h - x_N)_+^{2b}\big)\,dx_Nd\bm{x}' = \la^{N-1}\mathcal{I}_h(u).
\end{equation}
Thus
\begin{equation}
\label{infJinfI}
\inf\{\J_h(u) : u \in \K_{\ga}\} \le \la^{N-1}\inf\{\mathcal{I}_h(v) : v \in \K_{\ga,1\text{-d}}\}
\end{equation}
and consequently, to prove \Cref{exofgamma} we must show that for $\ga$ as in (\ref{gammarange}) the inequality above is a strict inequality.

Given $b,m,h > 0$, we let $g_h \colon \RR^+ \to \RR$ be defined by
\begin{equation}
\label{glam}
g_h(t) \coloneqq \frac{m^2}{t} + \frac{h^{2b+1} - (h - \min\{h,t\})^{2b+1}}{2b + 1}.
\end{equation}
Observe that $g_h \in C^1(\RR^+)$. Furthermore, for $t > 0$, we let $v_t \colon \RR^+ \to \RR$ be defined by
\begin{equation}
\label{1dsol}
v_t(s) \coloneqq m\left(1 - \frac{s}{t}\right)_+.
\end{equation}

\begin{thm}
\label{1dpb}
Given $b,m,h,\ga > 0$, let $\mathcal{I}_h$ and $\K_{\ga,1\text{-d}}$ be given as in $(\ref{1d})$ and $(\ref{K1d})$, respectively. Then, if $g_h$ and $v_t$ are given as above and the numbers $h^\#,h^*$ are defined as in (\ref{h*}), we have that
\begin{equation}
\label{I=g}
\inf\{\mathcal{I}_h(v) : v \in \K_{\ga,1\text{-d}}\} = \inf\{g_h(t) : 0 < t < \ga\}.
\end{equation}
Furthermore, the following hold:
\begin{itemize}
\item[$(i)$] if $h \le h^\#$ then $g_h$ is decreasing and $v_{\ga}$ is the only global minimizer of $\mathcal{I}_h$ in the class $\K_{\ga,1\text{-d}}$,
\item[$(ii)$] if $h^\# < h < h^* $ then $g_h$ has two critical points, $t_h,T_h$,
\[
0 < t_h < \frac{h}{b + 1} < T_h < h,
\]
which correspond to a point of local minimum and a point of local maximum of $g_h$, respectively. Moreover, there exists a unique $\tau_h > T_h$ such that $g_h(t_h) = g_h(\tau_h)$. In this case we have that
\begin{itemize}
\item[$(a)$] if $0 < \ga \le t_h$ then $g_h$ is decreasing in $(0,\ga)$ and $v_{\ga}$ is the only global minimizer of $\mathcal{I}_h$ in the class $\K_{\ga,1\text{-d}}$;
\item[$(b)$] if $t_h < \ga < \tau_h$ then $\inf\{\mathcal{I}_h(v) : v \in  \K_{\ga,1\text{-d}}\} = g_h(t_h)$ and $v_{t_h}$ is the only global minimizer of $\mathcal{I}_h$ in the class $\K_{\ga,1\text{-d}}$;
\item[$(c)$] if $\ga = \tau_h$ then $\inf\{\mathcal{I}_h(v) : v \in  \K_{\ga,1\text{-d}}\} = g_h(t_h) = g_h(\tau_h)$ and $v_{t_h},v_{\tau_h}$ are the only global minimizers of $\mathcal{I}_h$ in the class $\K_{\ga,1\text{-d}}$;
\item[$(d)$] if $\ga > \tau_h$ then $\inf\{\mathcal{I}_h(v) : v \in  \K_{\ga,1\text{-d}}\} = g_h(\ga)$ and $v_{\ga}$ is the only global minimizer of $\mathcal{I}_h$ in the class $\K_{\ga,1\text{-d}}$;
\end{itemize}
\item[$(iii)$] if $h \ge h^*$ then $t_h$ is a point of absolute minimum for $g_h$. Moreover, $v_{\ga}$ is the only global minimizer of $\mathcal{I}_h$ in the class $\K_{\ga,1\text{-d}}$ if $0 < \ga \le t_h$, while if $t_h < \ga$ then the only global minimizer is given by $v_{t_h}$.
\end{itemize}
\end{thm}

\begin{rmk}
\label{gaequiv}
Notice that $\ga$ is given as in $(\ref{gammarange})$ if and only if the following two conditions are simultaneously satisfied:
\begin{itemize}
\item[$(i)$] $g_h'(\ga) < 0$;
\item[$(ii)$] $\inf\{\mathcal{I}_h(v) : v \in  \K_{\ga,1\text{-d}}\} = g_h(\ga)$.
\end{itemize}
\end{rmk}

\begin{proof}[Proof of \Cref{1dpb}] We divide the proof into several steps. 
\newline
\textbf{Step 1:} By the direct method in the calculus of variations we have that there exists a global minimizer $v$ of $\mathcal{I}_h$ in $\K_{\ga,1\text{-d}}$. We claim that $v$ is linear on $\{v > 0\}$. Indeed, the minimality of $v$ implies that the set $\{v > 0\}$ is connected and the claim readily follows recalling that $v$ is harmonic in $\{v > 0\}$ (see \Cref{exofmin}). In turn, $v$ must be of the form $v = v_t$ for some $0 < t \le \ga$ and so (\ref{I=g}) follows upon noticing that 
\begin{equation}
\label{I=gpf}
\mathcal{I}_h(v_t) = g_h(t).
\end{equation}
Thus, it remains to study $\inf \{g_h(t) : 0 < t < \ga\}$.
\newline
\textbf{Step 2:} Since
\[
g'_h(t) = 
\left\{
\arraycolsep=1.4pt\def\arraystretch{2.3}
\begin{array}{ll}
\displaystyle -\frac{m^2}{t^2} + (h - t)^{2b} & \text{ if } t \le h, \\
\displaystyle -\frac{m^2}{t^2}  & \text{ if } t > h,
\end{array}
\right.
\]
we have that $g'_h(t) < 0$ if $t \ge h$. Moreover, $g'_h(t) \le 0$ for $t < h$ if and only if
\begin{equation}
\label{nolocmin}
\psi_h(t) \coloneqq - m^2 + t^2(h - t)^{2b} \le 0.
\end{equation}
Since $\psi_h$ has a global maximum in $(0,h)$ at the point $t = h/(b + 1)$, it follows that
\begin{equation}
\label{nolocmin2}
\psi_h(h/(b + 1)) = -m^2 + \frac{b^{2b}}{(b + 1)^{2b + 2}}h^{2b + 2} \le 0
\end{equation}
if and only if $h \le h^\#$, where $h^\#$ is the number given in $(\ref{h*})_1$. Consequently, if $h \le h^\#$ then $g_h$ is decreasing and so
\[
\inf\{g_h(t) : 0 < t < \ga\} = g_h(\ga),
\]
which, together with (\ref{I=g}) and (\ref{I=gpf}), shows that $v_{\ga}$ is the only global minimizer of $\mathcal{I}_h$ in the class $\K_{\ga,1\text{-d}}$.
\newline
\textbf{Step 3:} If $h > h^\#$ then, in view of (\ref{nolocmin}), (\ref{nolocmin2}), and the fact that $\psi_h$ has a single critical point in $(0,h)$, there exist 
\[
0 < t_h < \frac{h}{b + 1} < T_h < h
\]
such that $g_h$ strictly decreases in $(0,t_h)$ and in $(T_h, \infty)$, and strictly increases in $(t_h,T_h)$.
It follows that 
\begin{equation}
\label{(IV)}
\inf \{g_h(t) : 0 < t < \ga\} = 
\left\{
\arraycolsep=1.4pt\def\arraystretch{1.6}
\begin{array}{ll}
g_h(\ga) & \text{ if } 0 < \ga \le t_h, \\
g_h(t_h)  & \text{ if } t_h < \ga \le T_h, \\
\min \{g_h(t_h), g_h(\ga)\} & \text{ if } \ga > T_h.
\end{array}
\right.
\end{equation}
Hence, in what follows, it remains to treat the case $\ga > T_h$. Notice that
\begin{equation}
\label{(V)}
\inf \{g_h(t) : 0 < t < \ga\} = g_h(t_h) \le \lim_{t \to \infty}g_h(t) = \frac{h^{2b + 1}}{2b + 1}
\end{equation}
if and only if 
\[
m^2(2b+1) \le \sup \{f_h(t) : 0 < t < h\},
\]
where $f_h(t) \coloneqq t(h - t)^{2b + 1}$. The function $f_h$ has a maximum at $t = h/(2b+2)$, and so the previous condition reduces to
\[
m^2(2b + 1) \le f_h(h/(2b+2)),
\]
or equivalently $h \ge h^*$, where $h^*$ is the number given in $(\ref{h*})_2$.
Hence, it follows from (\ref{(V)}) that if $h \ge h^*$ then $g_h(t_h) < g_h(\ga)$, which, by (\ref{I=g}), (\ref{I=gpf}), and (\ref{(IV)}), proves $(iii)$. On the other hand, if $h^{\#} < h < h^*$, then by (\ref{(V)}) there exists $\tau_h > T_h$ such that $g_h(t_h) = g_h(\tau_h)$. 

Properties $(a),(b),(c),(d)$ now follow again by (\ref{I=g}),(\ref{I=gpf}), and (\ref{(IV)}).
\end{proof}

\begin{cor}
\label{thThtauh}
Let $t_h$, $T_h$, and $\tau_h$ be defined as in \Cref{1dpb}. Then, seen as functions of the variable $h$, $t_h$ is decreasing, $T_h$ is increasing, and $\tau_h$ is increasing.
\end{cor}
\begin{proof}
By the implicit function theorem we have that the maps $h \mapsto t_h$, $h \mapsto T_h$, and $h \mapsto \tau_h$ are differentiable, and we write $t_h'$, $T_h'$, and $\tau_h'$ to denote the derivatives. In particular, we see that for $h > h^{\#}$
\begin{align*}
t_h' & = -\frac{bt_h}{h - (b + 1)t_h} < 0, \\
T_h' & = -\frac{bT_h}{h - (b + 1)T_h} > 0.
\end{align*}
To prove the statement about $\tau_h$, we first assume that $\tau_h < h$. Recall that $\tau_h$ is defined through the identity
\[
\frac{m^2}{t_h} - \frac{(h - t_h)^{2b + 1}}{2b + 1} = \frac{m^2}{\tau_h} - \frac{(h - \tau_h)^{2b + 1}}{2b + 1};
\] 
differentiating both sides with respect to $h$ yields
\begin{equation}
\label{tauh'}
\frac{m^2t_h'}{t_h^2} + (h - t_h)^{2b}(1 - t_h') = \frac{m^2\tau_h'}{\tau_h^2} + (h - \tau_h)^{2b}(1 - \tau_h').
\end{equation}
The definition of $t_h$ can now be used to simplify the left-hand side of (\ref{tauh'}):
\[
\frac{m^2t_h'}{t_h^2} + (h - t_h)^{2b}(1 - t_h') = t_h'\left(\frac{m^2}{t_h^2} - (h - t_h)^{2b}\right) + h - t_h = h - t_h.
\]
Therefore we can rewrite (\ref{tauh'}) as
\[
\left(\frac{m^2}{\tau_h^2} - (h - \tau_h)^{2b}\right)\tau_h' = \tau_h - t_h,
\]
and the conclusion follows recalling that $t_h < \tau_h$ and $m^2 - \tau_h^2(h - \tau_h)^{2b} > 0$. The proof for the case $\tau_h \ge h$ is similar but simpler, therefore we omit it.
\end{proof}

\section{Existence of nontrivial minimizers} 
\label{nontriv}
In this section we present the proof of \Cref{exofgamma}.

\begin{proof}[Proof of \Cref{exofgamma}]
Let $\ga$ be as in (\ref{gammarange}). In view of (\ref{tonelliJ}), (\ref{infJinfI}), and condition $(ii)$ in \Cref{gaequiv}, it is enough to exhibit a function $w \in \K_{\ga}$ with the property that
\begin{equation}
\label{ntsga}
\J_h(w) < \la^{N - 1}g_h(\ga).
\end{equation}
Let $\delta$ be a positive real number, which we choose later, and define $\mathcal{S}$ to be the convex hull of $\R \times \{\ga\}$ with the point $\{(\bm{0},\ga + \delta)\}$, that is, the pyramid with base $\R \times \{\ga\}$ and vertex $\{(\bm{0},\ga + \delta)\}$. Define $\tilde{f} \colon \R \to \RR$ via
\[
\tilde{f}(\bm{x}') \coloneqq \sup\{t : (\bm{x}',t) \in \mathcal{S}\},
\]
and let $f$ be the periodic extension of $\tilde{f}$ to $\RR^{N - 1}$. We can then define 
\[
w(\bm{x}) \coloneqq m\left(1 - \frac{x_N}{f(\bm{x}')}\right)_+, \quad \bm{x} = (\bm{x}',x_N) \in \RR^N_+.
\]
The function $w$ defined as above belongs to class $\K_{\ga}$; furthermore, we claim that if $\delta$ is chosen sufficiently small, then $w$ satisfies (\ref{ntsga}). The proof of the claim is divided into several steps.
\newline
\textbf{Step 1: }In this step we study the asymptotic behavior of $\J_h(w)$ as $\delta \to 0^+$ with first-order accuracy. We do so by first noticing that if $\bm{x} \in \Omega$ is such that 
\begin{equation}
\label{xixj}
|x_i| \ge |x_j| \quad \text{ for some } i \in \{1, \dots, N - 1\} \text{ and every } j \le N - 1,
\end{equation}
then $w$ can be rewritten as follows:
\[
w(\bm{x}) = m\left(1 - \frac{x_N \la}{(\ga + \delta)\la - 2\delta|x_i|}\right)_+.
\]
In turn, if we denote by $\Omega_1$ the subset of $\Omega$ such that $x_j \ge 0$ for every $j$ and condition (\ref{xixj}) is satisfied for $i = 1$, we find that
\begin{equation}
\label{2n}
\int_{\Omega}|\nabla w(\bm{x})|^2\,d\bm{x} = 2^{N - 1}(N - 1)\int_{\Omega_1}|\nabla w(\bm{x})|^2\,d\bm{x}.
\end{equation}
Notice that $w$ restricted to $\Omega_1$ depends only on the variables $x_1$ and $x_N$; moreover, for $\mathcal{L}^{N}$-a.e. $\bm{x} \in \Omega_1$, we have that
\begin{align}
\frac{\partial w}{\partial x_1}(\bm{x}) & = 
\left\{
\arraycolsep=1.4pt\def\arraystretch{2.3}
\begin{array}{ll}
\displaystyle \frac{-2 m \la \delta x_N}{\left((\ga + \delta)\la - 2\delta x_1\right)^2} & \text{ if } x_N < \ga + \delta - \frac{2}{\la}\delta x_1, \\
\displaystyle 0  & \text{ otherwise, }
\end{array}
\right. \label{der1}
\\
\frac{\partial w}{\partial x_N}(\bm{x}) & = 
\left\{
\arraycolsep=1.4pt\def\arraystretch{2.3}
\begin{array}{ll}
\displaystyle \frac{-m\la}{(\ga + \delta)\la - 2\delta x_1} & \text{ if } x_N < \ga + \delta - \frac{2}{\la}\delta x_1, \\
\displaystyle 0  & \text{ otherwise. }
\end{array}
\right. \label{der2}
\end{align}
In view of (\ref{der1}) and (\ref{der2}), and by means of a direct direction computation, we see that
\begin{equation}
\label{intO1}
\int_{\Omega_1}|\nabla w(\bm{x})|^2\,d\bm{x} = m^2\left(\frac{4\delta^2}{3\la} + \la\right) \int_0^{\la/2}F(x_1, \delta)\,dx_1,
\end{equation}
where
\[
F(x_1, \delta) \coloneqq \frac{x_1^{N - 2}}{(\ga + \delta)\la - 2\delta x_1}.
\]
Notice that $F(x_1, \cdot)$ is a smooth function in a neighborhood of the origin, and that its first-order Taylor approximation centered at zero is given by
\[
F(x_1, \delta) = \frac{x_1^{N - 2}}{\ga \la} - \frac{x_1^{N - 2}(\la - 2x_1)}{\ga^2 \la^2}\delta + \mathcal{O}(\delta^2).
\]
Substituting the previous expansion into (\ref{intO1}) and combining the result with (\ref{2n}) we obtain that
\begin{align}
\label{graddelta}
\int_{\Omega}|\nabla w(\bm{x})|^2\,d\bm{x} & = 2^{N - 1}(N - 1)m^2\la \int_0^{\la/2}\left(\frac{x_1^{N - 2}}{\ga \la} - \frac{x_1^{N - 2}(\la - 2x_1)}{\ga^2 \la^2}\delta\right)\,dx_1 + \mathcal{O}(\delta^2) \notag \\
& = \frac{\la^{N - 1}m^2}{\ga} - \frac{\la^{N - 1}m^2}{N\ga^2}\delta + \mathcal{O}(\delta^2).
\end{align}
To compute the contribution coming from the area-term we distinguish between two cases. Indeed, if we assume that $h \le \ga$, we have that
\begin{equation}
\label{areah<g}
\int_{\Omega}\chi_{\{w > 0\}}(\bm{x})(h - x_N)^{2b}_+\,d\bm{x} = \la^{N - 1}\int_0^h(h - x_N)^{2b}\,dx_N = \la^{N - 1}\frac{h^{2b + 1}}{2b + 1}.
\end{equation}
On the other hand, if $\ga < h$ we take $\delta$ so small that $\ga + \delta \le h$. Then, reasoning as in (\ref{2n}), we have that
\begin{align}
\label{areag<h}
\int_{\Omega}\chi_{\{w > 0\}}(\bm{x})(h - x_N)^{2b}_+\,d\bm{x} & = 2^{N - 1}(N - 1)\int_{\Omega_1}\chi_{\{w > 0\}}(\bm{x})(h - x_N)^{2b}_+\,d\bm{x} \notag \\
& = \la^{N - 1}\frac{h^{2b + 1}}{2b + 1} - \frac{2^{N - 1}(N - 1)}{2b + 1}\int_0^{\la/2}G(x_1, \delta)\,dx_1,
\end{align}
where
\[
G(x_1, \delta) \coloneqq x_1^{N - 2}\left(h - \ga - \delta + \frac{2}{\la}\delta x_1\right)^{2b + 1}.
\]
Similarly to above, we consider the first-order Taylor approximation centered at zero for $G(x_1, \cdot)$, i.e.,
\[
G(x_1, \delta) = x_1^{N - 2}(h - \ga)^{2b + 1} + (2b + 1)x_1^{N - 2}(h - \ga)^{2b}\left(\frac{2}{\la}x_1 - 1\right)\delta + \mathcal{O}(\delta^2),
\]
and we substitute this expression into (\ref{areag<h}); by doing so we obtain
\begin{equation}
\label{ag<h}
\int_{\Omega}\chi_{\{w > 0\}}(\bm{x})(h - x_N)^{2b}_+\,d\bm{x} = \la^{N - 1}\left(\frac{h^{2b + 1}}{2b + 1} - \frac{(h - \ga)^{2b + 1}}{2b + 1} + \frac{(h - \ga)^{2b}}{N}\delta\right) + \mathcal{O}(\delta^2).
\end{equation}
\textbf{Step 2: }In this step we show that (\ref{ntsga}) is satisfied when $h \le \ga$. Indeed, recalling that $g_h$ is the function defined in (\ref{glam}), by (\ref{I=g}), (\ref{gammarange}), and \Cref{gaequiv} $(ii)$ we see that (\ref{ntsga}) is equivalent to 
\[
\J_h(w) < \la^{N - 1}g_h(\ga) = \la^{N - 1}\left(\frac{m^2}{\ga} + \frac{h^{2b+1}}{2b + 1}\right).
\]
To conclude, it is enough to notice that by (\ref{graddelta}) and (\ref{areah<g}) the previous condition reduces to
\[
- \frac{\la^{N - 1}m^2}{N\ga^2} + \mathcal{O}(\delta) < 0,
\]
which is satisfied provided $\delta$ is sufficiently small.
\newline
\textbf{Step 3: }In this final step we deal with the more delicate case in which $\ga < h$. Reasoning as above, we use (\ref{glam}), (\ref{graddelta}), and (\ref{areag<h}) to rewrite (\ref{ntsga}). By (\ref{gammarange}), (\ref{I=g}), and \Cref{gaequiv} $(ii)$ we have that (\ref{ntsga1}) reduces to
\[
\J_h(w) < \la^{N - 1}g_h(\ga) = \la^{N - 1} \left( \frac{m^2}{\ga} + \frac{h^{2b+1} - (h - \ga)^{2b+1}}{2b + 1}\right).
\]
Simplifying the terms that appear on both sides we are left to verify the following inequality:
\begin{equation}
\label{ntsga1}
- \frac{\la^{N - 1}m^2}{N\ga^2}\delta + \la^{N - 1}\frac{(h - \ga)^{2b}}{N}\delta + \mathcal{O}(\delta^2) < 0.
\end{equation}
Notice that the left-hand side of (\ref{ntsga1}) can be rewritten as
\[
\frac{\la^{N - 1}\delta}{N}\left(-\frac{m^2}{\ga^2} + (h - \ga)^{2b}\right) + \mathcal{O}(\delta^2) = \frac{\la^{N - 1}\delta}{N}g_h'(\ga) + \mathcal{O}(\delta^2).
\]
The desired inequality (\ref{ntsga1}), and therefore (\ref{ntsga}), follows from \Cref{gaequiv} $(i)$, provided $\delta$ is sufficiently small. This concludes the proof.
\end{proof}

\begin{rmk}
The result of \Cref{exofgamma} is optimal for $h < h^{\#}$ and $h \ge h^*$. However, it is still unclear whether the result could be improved for $h^{\#} \le h < h^*$.
\end{rmk}

\begin{figure}[h]
\centering
\begin{tikzpicture}[blend group=normal, scale=.7]

\draw [->] (-0.5,0) -- (12,0);
\draw [->] (0,-0.5) -- (0,8.1);
\draw [dashed] (3.5,0) -- (3.5,8);
\draw [dashed] (8,0) -- (8,8);

\draw [dashed] (0,2.5) -- (3.5,2.5);
\draw [dashed] (0,1.035) -- (8,1.035);

\draw (3.5,2.5) to [out=340,in=175] (11.9,0.4);
\draw (3.5,2.5) to [out=35,in=270] (7.97,8);

\draw [fill=black, fill opacity=0.2, draw opacity=0] (0,0) -- (3.5,0) -- (3.5,2.5) to [out=35,in=270] (7.97,8) -- (0,8);
\draw [fill=black, fill opacity=0.2, draw opacity=0] (3.5,0) -- (3.5,2.5) -- (3.5,2.5) to [out=340,in=175] (11.9,0.4) -- (11.9,0);
\draw [fill=black, fill opacity=0.2, draw opacity=0, pattern=north west lines] (3.5,2.5) to [out=340,in=165] (8,1.035) -- (8,8) -- (7.97,8) to [out=270,in=35] (3.5,2.5);

\node [below] at (3.5,0) {$h^{\#}$};
\node [below] at (8,0) {$h^*$};
\node [below] at (11.8,0) {$h$};
\node [left] at (0,7.8) {$\ga$};
\node [left] at (0,2.5) {$h^{\#}/(b + 1)$};
\node [left] at (0,1.035) {$h^*/(2b + 2)$};
\node [above] at (10,0.8) {$t_h$};
\node [above] at (6.5,5.5) {$\tau_h$};
\node [above] at (6.5,2.8) {?};
\end{tikzpicture}
\end{figure}

\begin{rmk}
We report here the explicit values of $t_h$ and $T_h$ for the case $b = 1/2$. As previously mentioned in the introduction to this paper, this case is of particular interest when $N = 2$ since it corresponds to Bernoulli-type free boundary problems related to water waves. For $0 < t < h$,
\[
g'_h(t) = -\frac{m^{2}}{t^{2}} + h - t.
\]
If $h > h^\#$, the cubic equation $t^{3} - h t^{2} + m^{2} = 0$ has three real solutions, two of which are positive. Setting
\[
\theta \coloneqq \arccos\left(  1-\frac{3^{3}}{2}\frac{m^{2}}{h^{3}}\right)
\]
so that $0<\theta<\pi$, the two positive solutions are given by
\begin{align*}
t_h  & :=\frac{2h}{3}\cos\frac{\theta+4\pi}{3}+\frac{h}
{3}\in\left(  0,\frac{2h}{3}\right) , \\ 
T_h  & :=\frac{2h}{3}\cos\frac{\theta}{3}+\frac{h}{3}
\in\left(  \frac{2h}{3},h\right).
\end{align*}
We also know that
\[
t_h < 2^{1/3}m^{2/3} < T_h.
\]
Indeed, for every $\eta \in (0, h - h^\#)$ \Cref{thThtauh} implies that
\[
t_h < t_{h - \eta} < \frac{2}{3}(h - \eta) < T_{h - \eta} < T_h.
\]
To conclude, let $\eta \to h - h^\#$.
\end{rmk}

Notice that \Cref{exofgamma} doesn't a priori exclude the existence of minimizers with a flat free boundary, i.e., minimizers whose free boundaries coincide with a horizontal hyperplane. The issue is addressed by the following corollary. 
\begin{cor}
\label{nolines}
For $\ga$ given as in $(\ref{gammarange})$, let $u \in \K_{\ga}$ be a global minimizer of $\J_h$. Then the free boundary $\partial \{u > 0\}$ does not coincide with a hyperplane of the form $\{x_N = k\}$, for some $k > 0$.
\end{cor}
\begin{proof}
Assume for the sake of contradiction that this is not the case; then $k \le h$. Assume first that $0 < k \le \ga$. We claim that 
\[
v(\bm{x}',x_N) = m\left(1 - \frac{x_N}{k}\right)_+
\]
satisfies $\J_h(v) = \J_h(u)$. Notice that since by assumption $k \le \ga$, we have that $v \in \K_{\ga}$. Hence, the claim would imply that $v$ is a global minimizer of $\J_h$, a contradiction to our choice of $\ga$. To prove the claim it is enough to observe that Tonelli's theorem and Jensen's inequality yield
\[
\begin{aligned}
\int_{\Omega}|\nabla u|^2\,d\bm{x} \ge \int_{\R}\int_0^k\left(\partial_{x_N} u\right)^2\,dx_Nd\bm{x}' \ge \int_{\R}\frac{1}{k}\left(\int_0^k \partial_{x_N}u\,dx_N\right)^2\,d\bm{x}' = \frac{\la^{N - 1} m^2}{k} = \int_{\Omega}|\nabla v|^2\,d\bm{x},
\end{aligned}
\]
and that the functions $u$ and $v$ have the same support. On the other hand, since the free boundary detaches tangentially from a smooth portion of the Dirichlet fixed boundary (see Theorem 1.1 in \cite{MR3916702}), we see also that it is not possible for $k$ to be larger than $\ga$, and the result is thus proved.
\end{proof}

\section{Properties of global minimizers}
\label{propertiesglobmin}
The aim of this section is to study qualitative properties of global minimizers of the functional $\J_h$ defined in (\ref{Jper}). In particular, our main interest lies in understanding how the shape of global minimizers is influenced by the parameter $h$. To this end, throughout the rest of this section for every $h > 0$ we make the following choice for the parameter $\ga$: 
\[
\gamma_h \coloneqq \theta(h), \quad \theta \colon \RR_+ \to \RR_+ \text{ non-increasing.}
\]
We then denote with $u_h$ solutions to the minimization problem for $\J_h$ in $\K_{\ga_h}$.

\subsection{Existence of a critical height}
\label{hcritsec}

To prove \Cref{exofhcrit} we begin by showing that $h_{\cri} < \infty$.
\begin{thm}[Existence of solutions with bounded support]
\label{regwaves}
Under the assumptions of \Cref{exofhcrit}, for every $\bar{x}_N  > 0$ such that 
\[
\bar{x}_N \neq \lim_{h \to \infty}\ga_h
\]
there exists $h_0 = h_0(b, m, \bar{x}_N, \la, \theta)$ such that if $h \ge h_0$ then the support of every global minimizer of $\J_h$ in $\K_{\ga_h}$ is contained in the set $\{x_N < h\}$.
\end{thm}

\begin{proof}
Let $\bar{x}_N > 0$ be given, and assume first that $\bar{x}_N > \lim_{h \to \infty} \ga_h$. Let 
\[
r \coloneqq \min\left\{\la, \frac{|\bar{x}_N - \lim_{h \to \infty} \ga_h|}{4}\right\},
\]
$h_1$ be such that $\ga_h \le \lim_h \ga_h + r$ for every $h \ge h_1$, and notice that
\[
B_r(\bm{x}',\bar{x}_N) \subset \{x_N > \ga_h\}
\]
for every $\bm{x}' \in \R$ and every $h \ge h_1$. Then, if $h > \bar{x}_N - r/2$ and $u_h \in \K_{\ga_h}$ is a global minimizer of $\J_h$, it follows from Lemma 2.3 in \cite{MR618549} that
\[
\frac{1}{r(h - \bar{x}_N - r/2)^b} \fint_{\partial B_r(\bm{x}',\bar{x}_N)}u_h\, d\mathcal{H}^{N - 1} \le \frac{m}{r(h - \bar{x}_N - r/2)^b}.
\]
Let $h_0 \ge h_1$ be such that
\[
\frac{m}{r(h - \bar{x}_N - r/2)^b} \le C(N,1/2),
\] 
where $C(N,1/2)$ is the constant in \Cref{competition2}. Then, for every $h \ge h_0$, we are in a position to apply \Cref{competition2} to conclude that $u_h$ is identically equal to zero in the set $\R \times (\bar{x}_N - r/2,\bar{x}_N + r/2)$. Since by minimality the support of $u_h$ is connected, it follows that $u_h$ must also vanish in $\R \times (\bar{x}_N,\infty)$. This concludes the proof in this case.

On the other hand, if $\lim_{h \to \infty}\ga_h > \bar{x}_N$, then it must be the case that 
\[
\overline{B_r(\bm{x}',\bar{x}_N)} \subset \{x_N < \ga_h\}
\]
for every $h > 0$, and thus we can proceed as above. 
\end{proof}


The following result is inspired by Theorem 10.1 in \cite{MR1009785} (see also Section 5 in \cite{MR647374} and Theorem 5.5 in \cite{MR2915865}).

\begin{thm}[Monotonicity]
\label{monotonicity} 
Under the assumptions of \Cref{exofhcrit}, consider $0 < d < h$ and let $u_d \in \K_{\ga_d}$ and $u_h \in \K_{\ga_h}$ be global minimizers of $\J_d$ and $\J_h$, respectively. Then
\begin{equation}
\label{mon1}
\{\bm{x} \in \Omega : u_h(\bm{x}) > 0\} \subset \{\bm{x} \in \Omega : u_d(\bm{x}) > 0\}
\end{equation}
and 
\begin{equation}
\label{mon2}
u_h \le u_d.
\end{equation}
Moreover, if there exists $\bm{x}_0 \in \partial \{u_h > 0\} \cap \Omega$ such that the free boundary is regular in a neighborhood of $\bm{x}_0$ then $u_h < u_d$ in $\{\bm{x} \in \Omega : u_d(\bm{x}) > 0\}$.
\end{thm} 

\begin{proof} 
We divide the proof into several steps.
\newline
\textbf{Step 1:} Define $v_1 \coloneqq \min\{u_d, u_h\}$ and $v_2 \coloneqq \max\{u_d, u_h\}$. Since by assumption $\ga_h$ is non-increasing as a function of $h$, we have that $v_1 \in \K_{\ga_h}$ and $v_2 \in \K_{\ga_d}$, and so
\begin{equation}
\label{v1v2min}
\J_d(u_d) + \J_h(u_h) \le \J_d(v_2) + \J_h(v_1).
\end{equation}
Notice that 
\begin{align*}
\int_{\Omega}\left(|\nabla v_1|^2 + |\nabla v_2|^2\right)\,d\bm{x} & = \int_{\{u_h > u_d\}}\left(|\nabla v_1|^2 + |\nabla v_2|^2\right)\,d\bm{x} + \int_{\{u_h \le u_d\}}\left(|\nabla v_1|^2 + |\nabla v_2|^2\right)\,d\bm{x} \\
& = \int_{\{u_h > u_d\}}\left(|\nabla u_d|^2 + |\nabla u_h|^2\right)\,d\bm{x} + \int_{\{u_h \le u_d\}}\left(|\nabla u_h|^2 + |\nabla u_d|^2\right)\,d\bm{x} \\
& = \int_{\Omega}\left(|\nabla u_d|^2 + |\nabla u_h|^2\right)\,d\bm{x}.
\end{align*}
Therefore we can rewrite (\ref{v1v2min}) canceling out the gradient terms, and by rearranging the remaining terms we obtain
\begin{equation}
\label{v1v2min2}
\int_{\{u_h > u_d\}}\left(\chi_{\{u_h > 0\}} - \chi_{\{u_d > 0\}}\right)\left((h - x_N)_+^{2b} - (d - x_N)_+^{2b}\right)\,d\bm{x} \le 0.
\end{equation}
Since the integrand in (\ref{v1v2min2}) is nonnegative, and also recalling that $u_d$ and $u_h$ are continuous in $\Omega$, we see that
\[
\left(\{u_h > 0\} \cap \{x_N < h\}\right) \cap \{u_h > u_d\} \subset \left(\{u_d > 0\} \cap \{x_N < h\}\right) \cap \{u_h > u_d\},
\]
which together with the fact that
\[
\{u_h > 0\} \cap \{u_h \le u_d\} \subset \{u_d > 0\} \cap \{u_h \le u_d\}
\]
yields
\begin{equation}
\label{mon1<h}
\{u_h > 0\} \cap \{x_N < h\} \subset \{u_d > 0\} \cap \{x_N < h\}.
\end{equation}
Notice that if the set $\{\bm{x} \in \Omega : u_h(\bm{x}) > 0\}$ is contained in $\R \times (0,d)$ then (\ref{mon1}) follows from (\ref{mon1<h}). On the other hand, if this is not the case, again by (\ref{mon1<h}) we deduce the existence of a point $\bm{x} \in \R \times [d,\infty)$ with the property that $u_d(\bm{x}) > 0$. Since $u_d$ is harmonic in $\R \times \{x_N > d\}$, it must be the case that $u_d > 0$ in $\R \times (d,\infty)$. This conclude the proof of (\ref{mon1}).
\newline
\textbf{Step 2:} We observe that since (\ref{v1v2min2}) is actually an equality, then (\ref{v1v2min}) must an equality as well, and so $v_1$ and $v_2$ are global minimizers of $\J_h$ and $\J_d$ in $\K_{\ga_h}$ and $\K_{\ga_d}$, respectively. We now claim that if there is $\bm{x}_0 \in \Omega$ such that $u_d(\bm{x}_0) = u_h(\bm{x}_0) > 0$, then $u_d = u_h$ everywhere in $\Omega$. To see this, we notice that in a neighborhood of $\bm{x}_0$ the functions $u_d - v_2$ and $u_h - v_2$ are harmonic, nonpositive and attain a maximum at an interior point. Then, by the maximum principle, both $u_d - v_2$ and $u_h - v_2$ must vanish in the connected component of $\{u_h > 0\}$ which contains $\bm{x}_0$; the claim follows upon recalling that the set $\{u_h > 0\}$ is connected as a consequence of the minimality of $u_h$.

To prove (\ref{mon2}), assume by contradiction that there is $\bm{x} \in \Omega$ such that $u_h(\bm{x}) > u_d(\bm{x})$. If there is $\bm{y} \in \{u_h > 0\}$ such that $u_d(\bm{y}) > u_h(\bm{y})$, then by the connectedness of $\{u_h > 0\}$, together with the fact that $u_h$ and $u_d$ are continuous, we have that there is $\bm{z} \in \Omega$ such that $u_h(\bm{z}) = u_d(\bm{z}) > 0$. By the claim we just proved, this would imply that $u_h = u_d$, a contradiction. Hence $u_d \le u_h$ in $\{u_h > 0\}$, which together with (\ref{mon1}) implies that
\begin{equation}
\label{mon2pf}
 \{u_h > 0\} = \{u_d > 0\}.
\end{equation}
In turn, 
\begin{align}
\begin{aligned}
\label{areauhud}
\int_{\Omega}\chi_{\{u_h > 0\}}(h - x_N)_+^{2b}\,d\bm{x} & = \int_{\Omega}\chi_{\{u_d > 0\}}(h - x_N)_+^{2b}\,d\bm{x}, \\
\int_{\Omega}\chi_{\{u_d > 0\}}(d - x_N)_+^{2b}\,d\bm{x} & = \int_{\Omega}\chi_{\{u_h > 0\}}(d - x_N)_+^{2b}\,d\bm{x}.
\end{aligned}
\end{align}
From (\ref{mon2pf}) we also see that $u_d \in \K_{\ga_h}$. Consequently, we have that $\J_h(u_h) \le \J_h(u_d)$ and $\J_d(u_d) \le \J_d(u_h)$, which together with (\ref{areauhud}) imply that
\[
\int_{\Omega}|\nabla u_h|^2\,d\bm{x} = \int_{\Omega}|\nabla u_d|^2\,d\bm{x}.
\]
Consider $v \coloneqq \frac{1}{2} u_h + \frac 12 u_d \in \K_{\ga_h}$. By the strict convexity of the Dirichlet energy, we have
\[
\J_h(v) < \int_{\Omega}\left(\frac{1}{2}|\nabla u_h|^2 + \frac{1}{2}|\nabla u_d|^2 + \chi_{\{v > 0\}}(h - x_N)_+^{2b}\right)\,d\bm{x} = \J_h(u_h),
\]
a contradiction to the minimality of $u_h$, and (\ref{mon2}) is hence proved.
\newline
\textbf{Step 3:} Finally, assume by contradiction that there is $\bm{x}_1 \in \{u_d > 0\}$ such that $u_h(\bm{x}_1) = u_d(\bm{x}_1)$, so that $u_h = u_d$ in $\Omega$. Then, by \Cref{exofmin}, for $\bm{x}_0 = (\bm{x}_0',x_N)$ as in the statement we have that
\[
(h - x_N)^b = \partial_{-\nu} u_h(\bm{x}_0) = \partial_{-\nu} u_d(\bm{x}_0) = (d - x_N)^b. 
\]
This is in contradiction with the assumption $d \neq h$. Hence $u_h < u_d$ in $\{u_d > 0\}$, as claimed.
\end{proof}

\begin{proof}[Proof of \Cref{exofhcrit}]
Let 
\begin{multline}
\label{hcritdef}
h_{\cri} \coloneqq \inf \{h > 0 : \exists \text{ a global minimizer } u_h \in \K_{\ga_h} \text{ of } \J_h \text{ such that } \supp u_h \subset \{x_N \le h\}\}.
\end{multline}
By \Cref{regwaves} we have that $h_{\cri} < \infty$. Assume for the sake of contradiction that $h_{\cri} = 0$. Then for every $h > 0$ there exists a global minimizer $u_h \in \K_{\ga_h}$ with the property that the support of $u_h$ is contained in the set $\{x_N \le h \}$. Reasoning as in the proof of \Cref{nolines}, we see that 
\begin{equation}
\label{trivialbound}
\J_h(u_h) > \int_{\Omega}|\nabla u_h|^2\,d\bm{x} \ge \frac{\la^{N - 1} m^2}{h}.
\end{equation}
Since by assumption the function $\theta$ is non-increasing, there exists $\bar{h}$ such that if $h \le \bar{h}$ then
\[
h \le \theta(h) = \ga_h.
\]
For every such $h$ we let $w$ be the function defined in the proof of \Cref{exofgamma}. Then, it follows from (\ref{graddelta}) and (\ref{areah<g}) that 
\[
\J_h(w) < \frac{\la^{N - 1}m^2}{\ga_h} + \la^{N - 1}\frac{h^{2b + 1}}{2b + 1} + \mathcal{O}(\delta).
\]
Notice that if $h$ is chosen sufficiently small 
\[
\frac{\la^{N - 1}m^2}{\ga_h} + \la^{N - 1}\frac{h^{2b + 1}}{2b + 1} < \frac{\la^{N - 1}m^2}{h}.
\]
In turn, by (\ref{trivialbound}) for every $\delta$ small enough we see that
\[
\J_h(w) < \J_h(u_h).
\] 
Since by definition $w \in \K_{\ga}$, this contradicts the minimality of $u_h$. Consequently, we have shown that $h_{\cri} > 0$. Properties $(ii)$ and $(iii)$ follow from \Cref{monotonicity}; we omit the details.
\end{proof}

\begin{rmk}
\label{intfbreg}
By \Cref{exofmin}, it follows that if $u \in \K_{\ga_h}$ is a global minimizer of $\J_h$ for $h > h_{\cri}$, then for every $K$ compactly contained in $\Omega$, the free boundary $\partial\{u > 0\} \cap K$ is a smooth hypersurface except possibly on a closed singular of Hausdorff dimension $N - 5$.
\end{rmk}

\subsection{Scaling of the critical height}
\begin{thm}[Comparison principle]
\label{cp} 
Given $b, m, h, \delta, \ga, \la > 0$, let $u$ be a global minimizer of $\J_h$ in $\K_{\delta}$ and let $w$ be a global minimizer of $\J_h$ in $\K_{\ga}$, where $\J_h$ is the functional in $(\ref{Jper})$ and $\K_{\delta}, \K_{\ga}$ are defined as in $(\ref{Kg})$. Then either 
\[
\{u > 0\} \subset \{w > 0\}\ \text{and}\ u \le w\
\]
or
\[
\{w > 0\} \subset \{u > 0\}\ \text{and}\ w \le u.
\]
\end{thm}

\begin{proof} Assume without loss of generality that $\delta \le \ga$. As in the proof of \Cref{monotonicity}, we consider $v_1 \coloneqq \min\{u,w\}$ and $v_2 \coloneqq \max\{u,w\}$. Then $v_1 \in \K_{\delta}$, $v_2 \in \K_{\ga}$, and in particular we have   
\[
\J_h(u) + \J_h(w) = \J_h(v_1) + \J_h(v_2).
\]
Therefore $v_1$ and $v_2$ are global minimizers of $\J_h$ in $\K_{\delta}$ and in $\K_{\ga}$, respectively. Reasoning as in the proof of \Cref{monotonicity}, we recall that if there exists a point $\bm{x}_0$ such that $u(\bm{x}_0) = w(\bm{x}_0) > 0$ then $u = w$ everywhere in $\Omega$. Next, we assume by contradiction that the supports of $u$ and $w$ do not satisfy the inclusions as in the statement, i.e.\@, there exist $\bm{x}, \bm{y} \in \Omega$ such that $u(\bm{x}) > 0$, $w(\bm{y}) > 0$ and $u(\bm{y}) = w(\bm{x}) = 0$. Let $\bm{z} \in \Omega$ be such that $u(\bm{z}) > 0$ and $w(\bm{z}) > 0$ (such a point $\bm{z}$ exists since by minimality we have that $\J_h(u)$ and $\J_h(w)$ are both finite). We assume first that $w(\bm{z}) > u(\bm{z})$. Then, since by minimality $\{u > 0\}$ is open and connected and thus path-wise connected, we can find a continuous curve $\bm{\varphi} \colon [0,1] \to \Omega$ joining $\bm{z}$ to $\bm{x}$, with support contained in $\{u > 0\}$. Define 
\[
v(t) \coloneqq w(\bm{\varphi}(t)) - u(\bm{\varphi}(t)).
\]
Notice that by construction $v(0) = w(\bm{z}) - u(\bm{z}) > 0$ and $v(1) = w(\bm{x}) - u(\bm{x}) < 0$, and so there exists $t_0 \in (0,1)$ such that $v(t_0) = 0$. Thus $0 < u(\bm{\varphi}(t)) = w(\bm{\varphi}(t))$, which in turn implies that $u = w$, a contradiction. Similarly, if $u(\bm{z}) > w(\bm{z})$, we arrive to a contradiction by considering a continuous curve $\bm{\psi} \colon [0,1] \to \Omega$ that joins $\bm{z}$ with $\bm{y}$ and with support contained in $\{w > 0\}$. The rest of the proof is analogous to the proof of (\ref{mon2}).
\end{proof}

\begin{lem}
\label{hc<}
Under the assumptions of \Cref{exofhcrit}, we have that 
\[
h_{\cri} \le h^* = \frac{2b + 2}{(2b + 1)^{b/(b + 1)}}m^{1/(b + 1)}.
\]
\end{lem}
\begin{proof}
Assume by contradiction that $h_{\cri} > h^*$, and let $h^* < h < h_{\cri}$. By Tonelli's theorem and \Cref{1dpb} $(iii)$ we have that the function $w \colon \RR^N_+ \to \RR$ defined by 
\[
w(\bm{x}) \coloneqq v_{t_h}(x_N),
\]
is the unique global minimizer of $\J_h$ in $\K_{\ga}$ for every $\ga \ge t_h$. Notice that by (\ref{gammarange}) it must be the case that $\ga_h < t_h$. Let $u_h \in \K_{\ga_h}$ be a global minimizer of $\J_h$. Since by assumption $u(\bm{x}) = 0$ for $\bm{x} = (\bm{x}',x_N) \in \partial \R \times (\ga_h, \infty)$, by continuity we can find $\bm{x}'_0 \in \R$ close to $\partial \R$ such that 
\[
u(\bm{x}'_0, \ga_h) < m\left(1 - \frac{\ga_h}{t_h}\right) = w(\bm{x}'_0, \ga_h).
\]
Then it follows from \Cref{cp} that $u_h \le w$, and in particular 
\[
\{u_h > 0\} \subset \{w > 0\} = \{x_N < t_h\}.
\]
Hence $u_h$ has bounded support in $\Omega$, we have reached a contradiction to the definition of $h_{\cri}$ (see \Cref{exofhcrit}).
\end{proof}

\begin{lem}
\label{scaling2}
Under the assumptions of \Cref{exofhcrit}, the following hold:
\begin{itemize}
\item[$(i)$] let $a,c \in \RR_+$ be such that 
\begin{equation}
\label{ac}
0 < \frac{a}{1 - B(b)a^{2b + 2}} \le c, \quad \text{ where } B(b) \coloneqq \frac{(2b + 2)^{2b + 2}}{(2b + 1)^{2b + 1}},
\end{equation}
and define $h \coloneqq ah^*$. Then $h_{\cri} \ge h$ provided that $\ga_h \ge ch^*$;
\item[$(ii)$] if $\ga_{t_{h^*}} \ge t_{h^*}$ then
\[
h_{\cri} \ge \frac{m^{1/(b + 1)}}{(2b + 1)^{b/(b + 1)}}.
\]
\end{itemize}
\end{lem}
\begin{proof}
Let $a,c,h,\ga_h$ be as in statement $(i)$ and assume for the sake of contradiction that $h_{\cri} < h$. Then it follows from the definition of $h_{\cri}$ that there is a global minimizer $u_h \in \K_{\ga_h}$ of $\J_h$ with $\supp u_h \subset \{x_N  \le h\}$. As in the proof of \Cref{nolines},
\begin{equation}
\label{enlb}
\J_h(u_h) > \frac{\la^{N - 1}m^2}{h}.
\end{equation}
Let $w(\bm{x}) \coloneqq v_{\ga_h}(x_N)$, where $v_{\ga_h}$ is defined in (\ref{1dsol}). We claim that $\J_h(w) < \J_h(u_h)$. Since this would clearly be a contradiction to the minimality of $u_h$, the claim implies the desired result, i.e., $h_{\cri} \ge h$.
In view of the fact that by (\ref{tonelliJ}) and (\ref{I=g})
\[
\J_h(w) \le \frac{\la^{N - 1}m^2}{\ga_h} + \la^{N - 1}\frac{h^{2b + 1}}{2b + 1},
\]
and recalling (\ref{enlb}), to prove the claim it is enough to show that 
\[
\frac{m^2}{\ga_h} + \frac{h^{2b + 1}}{2b + 1} \le \frac{m^2}{h},
\]
which in turn is implied by (\ref{ac}).

To prove the second statement, we begin by showing that for every $h \le t_{h^*}$ 
\begin{equation}
\label{finallyabound}
\{y < t_{h^*}\} \subset \{u_h > 0\},
\end{equation}
where, as usual, $u_h$ refers to a global minimizer of $\J_h$ in $\K_{\ga_h}$. This fact follows from the simple observation that for $h \le t_{h^*}$, minimizers of $\J_h$ in $\K_{\ga_h}$ are independent of the values of $\ga_h$ for $h > t_{h^*}$. Consequently, we can assume without loss of generality that $\ga_h \ge t_{h^*}$ for every $h > t_{h^*}$. In particular, this implies that $\ga_{h^*} \ge t_{h^*}$ and therefore $u_{h^*}(\cdot,x_N) \coloneqq v_{t_{h^*}}(x_N)$ is the unique global minimizer of $\J_{h^*}$ (see \Cref{1dpb}). The rest follows from \Cref{monotonicity}. Assume for the sake of contradiction that 
\[
h_{\cri} < \frac{m^{1/(b + 1)}}{(2b + 1)^{b/(b + 1)}} = \frac{h^*}{2b + 2} = t_{h^*}.
\]
Notice that (\ref{finallyabound}) allows us to obtain the following refined version of (\ref{enlb}):
\begin{equation}
\label{eboundbelow}
\J_h(u_h) \ge \frac{\la^{N - 1}m^2}{h} + \la^{N - 1}\frac{h^{2b + 1}}{2b + 1}, \quad \text{ for }h \le t_{h^*}.
\end{equation}
Reasoning as above, by letting $w(\cdot,x_N) \coloneqq v_{\ga_h}(x_N)$, using the fact that $\ga_{t_{h^*}} \ge t_{h^*}$, and (\ref{eboundbelow}) we see that 
\[
\J_{t_{h^*}}(w) = \frac{\la^{N - 1}m^2}{\ga_{t_{h^*}}} + \la^{N - 1}\frac{t_{h^*}^{2b + 1}}{2b + 1} \le \frac{\la^{N - 1}m^2}{t_{h^*}} + \la^{N - 1}\frac{t_{h^*}^{2b + 1}}{2b + 1} \le \min \left \{\J_{t_{h^*}}(u) : u \in \K_{\ga_{t_{h^*}}}\right \}.
\]
In particular, since $w \in \K_{\ga_{t_{h^*}}}$, it must be the case that $w$ is a global minimizer of $\J_h$, a contradiction to \Cref{exofgamma}. This concludes the proof.
\end{proof}

\begin{rmk}
Inequality (\ref{ac}) holds for 
\[
0 < a \le \frac{1}{\left(2B(b)\right)^{1/(2b + 2)}}, \quad c \ge \frac{2}{\left(2B(b)\right)^{1/(2b + 2)}}.
\]
\end{rmk}

\Cref{scalingthm} is then an immediate corollary of \Cref{hc<} and \Cref{scaling2}.

\subsection{Structural properties of global minimizers}
In this subsection we present the proof of \Cref{structure}. For the clarity of presentation, the proof of is divided into a number of separate results. 
\begin{thm}
\label{convofminJ}
Under the assumptions of \Cref{exofhcrit}, if $\theta$ is right-continuous at $h > 0$, there exists $u_h^- \in \K_{\ga_h}$, a global minimizer of the functional $\J_h$, with the property that for every strictly decreasing sequence $\{h_n\}_n$ with $h_n \searrow h$ and for every sequence $\{u_n\}_n$ such that $u_n \in \K_{\ga_{h_n}}$ is a global minimizer of $\J_{h_n}$ for every $n \in \NN$
\[
\arraycolsep=1.4pt\def\arraystretch{1.6}
\begin{array}{rll}
\nabla u_n \to & \nabla u_h^- & \text{ in } L^2(\Omega;\RR^N), \\
u_n \to & u_h^- & \text{ in } H^1_{\loc}(\Omega), \\
u_n \to & u_h^- & \text{ uniformly on compact subsets of } \Omega.
\end{array}
\]
Similarly, if $\theta$ is left-continuous at $h$, there exists a global minimizer of $\J_h$ in $\K_{\ga_h}$, denoted by $u_h^+$, which enjoys analogous properties for strictly increasing sequences converging to $h$.
\end{thm}

We begin by proving a preliminary lemma. 
\begin{lem}
\label{approxinkappa}
Under the assumptions of \Cref{exofhcrit}, let $w \in \K_{\ga_h}$ be such that $\J_h(w) < \infty$. Then, if $\theta$ is right-continuous at $h > 0$, for every sequence $h_n \searrow h$ there is a corresponding sequence $\{w_n\}_n$ such that $w_n \in \K_{\ga_{h_n}}$ for every $n \in \NN$ and $\J_{h_n}(w_n) \to \J_h(w)$ as $n \to \infty$.
\end{lem}

\begin{proof}
Set
\[
\sigma_n \coloneqq \frac{\ga_h}{\ga_{h_n}},
\]
and define the rescaled functions $w_n(\bm{x}',x_N) \coloneqq w(\bm{x}',\sigma_n x_N)$. Notice that $w_n \in \K_{\ga_{h_n}}$ and that by a change of variables
\begin{align*}
\int_{\Omega}|\nabla w_n|^2\, d\bm{x} = &\ \int_{\Omega} \left(|\nabla_{\bm{x}'} w(\bm{x}',\sigma_n x_N)|^2 + \sigma_n^2 \partial_{x_N} w(\bm{x}',\sigma_n x_N)^2\right)\, d\bm{x} \\
= &\ \int_{\Omega} \left(|\nabla_{\bm{x}'} w(\bm{x}',z)|^2 + \sigma_n^2 \partial_{x_N} w(\bm{x}',z)^2\right)\sigma_n^{-1}\, d\bm{x}'dz \\
\to &\ \int_{\Omega}|\nabla w(x,z)|^2\, d\bm{x}'dz,
\end{align*}
where in the last step we have used the fact that by assumption $\sigma_n \searrow 1$. Similarly, one can show that
\[
\int_{\Omega} \chi_{\{w_n > 0\}}(h_n - x_N)_+\, d\bm{x} \to \int_{\Omega}\chi_{\{w > 0\}}(h - x_N)_+\, d\bm{x},
\]
and the result follows.
\end{proof}

\begin{proof}[Proof of \Cref{convofminJ}]
We divide the proof into several steps.
\newline
\textbf{Step 1:} 
Assume first that $\theta$ is right-continuos at $h$ and let $\{h_n\}_n$ and $\{u_n\}_n$ be given as in the statement. We begin by showing that there exists a subsequence of $\{u_n\}_n$ that converges to a function $u_h^- \in \K_{\ga_h}$. To this end, let $v \colon \RR^N_+ \to \RR$ be defined by
\[
v(\cdot,x_N) \coloneqq m\left(1 - \frac{x_N}{\ga_{h_1}}\right)_+
\] 
Then $v \in \K_{\ga_{h_n}}$ for every $n \in \NN$ and in particular we have the following chain of inequalities:
\[
\int_{\Omega}|\nabla u_n|^2\, d\bm{x} \le \J_{h_n}(u_n) \le \J_{h_n}(v) \le \J_{h_1}(v) < \infty.
\] 
Hence $\{\nabla u_n\}_n$ is bounded in $L^2(\Omega;\RR^N)$. Moreover, since $u_n - v = 0$ on $\R \times \{0\}$, by Poincar\'{e}'s inequality we obtain
\[
\int_{\Omega_r}|u_n - v|^2\, d\bm{x} \le C(\Omega_r)\int_{\Omega_r}|\nabla u_n - \nabla v|^2\, d\bm{x},
\]
where $\Omega_r \coloneqq \Omega \cap \{x_N < r\}$, with $r > 0$. This shows that $\{u_n\}_n$ is bounded in $H^1(\Omega_r)$ and thus, up to the extraction of a subsequence, $u_n \rightharpoonup u^r$ in $H^1(\Omega_r)$. If we now let $s > r$, eventually extracting a further subsequence, we have that $u_n \rightharpoonup u^r$ in $H^1(\Omega_r)$ and $u_n \rightharpoonup u^s$ in $H^1(\Omega_s)$. By the uniqueness of the weak limit we conclude that
\[
u^r(\bm{x}) = u^s(\bm{x}) \quad \text{ for } \mathcal{L}^N\text{-a.e. } \bm{x} \in \Omega_r. 
\]
By letting $r \nearrow \infty$ and by a diagonal argument, up to the extraction of consecutive subsequences, this defines a function $u_h^-$ such that for some $\{n_k\}_k \subset \NN$
\begin{equation}
\label{convunk}
\arraycolsep=1.4pt\def\arraystretch{1.6}
\begin{array}{rll}
\nabla u_{n_k} \rightharpoonup & \nabla u_h^- & \text{ in } L^2(\Omega;\RR^N), \\
u_{n_k} \to & u_h^- & \text{ in } L^2_{\loc}(\Omega), \\
u_{n_k} \to & u_h^- & \text{ pointwise a.e.\@ in } \Omega, \\
u_{n_k} \to & u_h^- & \text{ in } L^2_{\loc}(\partial \Omega).
\end{array}
\end{equation}
In particular, this shows that $u_h^-$ can be extended to a function in $\K_{\ga_h}$. 
\newline
\textbf{Step 2:} Next, we show that $u_h^-$ is a global minimizer of $\J_h$. We do so by first showing that up to the extraction of a subsequence which we don't relabel, 
\[
\chi_{\{u_{n_k} > 0\}} \stackrel{\ast}{\rightharpoonup} \xi \quad \text{ in }L^{\infty}(\Omega),
\]
where the function $\xi$ satisfies
\begin{equation}
\label{xi>chi}
\xi(\bm{x}) \ge \chi_{\{u_h^- > 0\}}(\bm{x}) \quad \text{ for } \mathcal{L}^N\text{-a.e. } \bm{x} \in \Omega.
\end{equation}
Indeed, arguing as in the proof of Theorem 1.3 in \cite{MR618549}, we observe that for every $K$ compactly contained in $\{u > 0\}$
\[
0 = \int_K \left(\chi_{\{u_{n_k} > 0\}} - 1\right)u_{n_k}\,d\bm{x} \to \int_K(\xi - 1)u_h^-\,d\bm{x}.
\]
Since $u_h^- > 0$ in $K$, then necessarily $\xi(\bm{x}) = 1$ for $\mathcal{L}^N$-a.e. $\bm{x} \in K$ and hence, by exhaustion, in $\{u > 0\}$. To prove that $u_h^-$ is a global minimizer of $\J_h$ in $\K_{\ga_h}$, we fix $r > 0$ and let $w \in \K_{\ga_h}$. If $\J_h(w) = \infty$ there is nothing to do, hence we assume without loss that $\J_h(w) < \infty$ and consider $\{w_n\}_n$ as in \Cref{approxinkappa}. Then we have
\begin{align}
\label{uglobmin}
\int_{\Omega_r}\left(|\nabla u_h^-|^2 + \chi_{\{u_h^- > 0\}}(h - x_N)_+\right)\,d\bm{x} \le &\ \int_{\Omega}\left(|\nabla u_h^-|^2 + \xi(h - x_N)_+\right)\,d\bm{x} \notag \\
\le &\ \liminf_{k \to \infty} \J_{h_{n_k}}(u_{n_k}) \le \lim_{k \to \infty}\J_{h_{n_k}}(w_{n_k}) \\
= &\ \J_h(w). \notag
\end{align}
We then conclude that $\J_h(u) \le \J_h(w)$ for every $w \in \K_{\ga_h}$ by letting $r \nearrow \infty$.
\newline
\textbf{Step 3:} Notice that taking $w = u_h^-$ in (\ref{uglobmin}) yields 
\[
\int_{\Omega_r}\left(|\nabla u_h^-|^2 + \chi_{\{u_h^- > 0\}}(h -x_N)_+\right)\, d\bm{x} \le \liminf_{k \to \infty}\J_{h_{n_k}}(u_{n_k}) \le \limsup_{k \to \infty}\J_{h_{n_k}}(u_{n_k}) \le \J_h(u_h^-).
\]
In turn, by letting $r \nearrow \infty$ we obtain
\begin{equation}
\label{limitenergy}
\J_h(u_h^-) = \lim_{k \to \infty}\J_{h_{n_k}}(u_{n_k}).
\end{equation}
On the other hand, by the lower semicontinuity of the $L^2$-norm and (\ref{xi>chi}) we see that
\[
\int_{\Omega}|\nabla u_h^-|^2\, d\bm{x} \le \liminf_{k \to \infty}\int_{\Omega}|\nabla u_{n_k}|^2\, d\bm{x},
\]
and 
\[
\int_{\Omega}\chi_{\{u_h^- > 0\}}(h - x_N)_+\, d\bm{x} \le \liminf_{k \to \infty}\int_{\Omega}\chi_{\{u_{n_k} > 0\}}(h - x_N)_+\, d\bm{x}.
\]
In view of (\ref{limitenergy}), we notice that the previous two inequalities are necessarily equalities and therefore 
\[
\nabla u_{n_k} \to \nabla u_h^- \quad \text{ in } L^2(\Omega;\RR^N).
\]
We recall that, by \Cref{monotonicity}, $\{u_{n_k}\}_k$ is an increasing sequence of continuous functions with a continuous pointwise limit (see (\ref{convunk})). Hence, by Dini's convergence theorem, the convergence is uniform on compact subsets of $\Omega$. This shows that eventually extracting a subsequence, the sequence $\{u_n\}_n$ converges in the desired fashion to a minimizer of $\J_h$.
\newline
\textbf{Step 4:} Suppose by contradiction that the entire sequence $\{u_n\}_n$ does not converge to $u_h^-$ as in the statement of the theorem, and let $\{u_{n_j}\}_j$ be a subsequence for which this fails. Applying the results of the previous steps to $\{u_{n_j}\}$ we can extract a further subsequence (which we don't relabel) which converges uniformly on compact subsets of $\Omega$ to a function $w \in \K_{\ga_h}$ which is by assumption different form $u_h^-$ and which is also a minimizer of $\J_h$. Consequently, it follows from \Cref{monotonicity} that $u_{n_k} \le w$ and $u_{n_j} \le u_h^-$. Let $x$ and $r$ be such that $B_r(\bm{x})$ is compactly contained in the support of $u_h^-$. Then, passing to the limit as $k \to \infty$ and $j \to \infty$ in the previous inequalities we obtain $u_h^- = w$ in $B_r(\bm{x})$ and in particular that $0 < u(\bm{x}) = w(\bm{x})$. Reasoning as in the proof of \Cref{monotonicity} we obtain that $u = w$ in $\Omega$.

Notice that the same technique can be used to show that $u_h^-$ is independent of the sequences $\{h_n\}_n$ and $\{u_n\}_n$, i.e., it only depends on the type of monotonicity. This concludes the proof of the first part.
\newline
\textbf{Step 5:} Assume that $\theta$ is left-continuous at $h$. Notice that the analogous result to \Cref{approxinkappa} is trivial in this case since the monotonicity assumption on $\theta$ guarantees that $\K_{\ga_h} \subset \K_{\ga_{h_n}}$ for every $n$. Indeed, for every $w \in \K_{\ga_h}$ we can set $w_n \coloneqq w$ and obtain immediately that $\J_h(w_n) \to \J_h(w)$. On the other hand, the additional assumption on $\theta$ is required to conclude from (\ref{convunk}) that $u_h^+$ belongs to the class $\K_{\ga_h}$. The rest follows essentially without changes and therefore we omit the details.
\end{proof}

The following result is an immediate corollary of \Cref{monotonicity} and \Cref{convofminJ}.
\begin{cor}
\label{upm}
Under the assumptions of \Cref{exofhcrit}, if $\theta$ is continuous at $h > 0$, there are two (possibly equal) global minimizers of $\J_h$ in $\K_{\ga_h}$, namely $u_h^-, u_h^+$, such that $u_h^- \le u_h^+$ and if $w$ is another global minimizer then $u_h^- \le w \le u_h^+$. 
\end{cor}

\begin{thm}[Uniqueness]
\label{unicitamin}
Under the assumptions of \Cref{exofhcrit}, if in addition $\theta$ is continuous, the functional $\J_h$ admits a unique global minimizer in $\K_{\ga_h}$ for all but countably many values of $h$.
\end{thm}
\begin{proof}
Let
\[
\Lambda \coloneqq \{h \in \RR_+ : \text{the minimization problem for } \J_h \text{ in } \K_{\ga_h} \text{ has at least two distinct solutions}\},
\]
for every integer $j \ge 2$ let $u_j \in \K_{\ga_j}$ be a global minimizer of $\J_j$, and denote by $B_j$ a ball compactly contained in the set $\{\bm{x} \in \Omega : u_j(\bm{x}) > 0\}$. Furthermore, for every $n \in \NN$ define the sets
\[
\Lambda_{j,n} \coloneqq \left\{h \in \left(\frac{1}{j},j\right) : \sup\left\{|u_h^+(\bm{x}) - u_h^-(\bm{x})| : \bm{x} \in B_j\right\} \ge \frac{1}{n}\right\},
\]
where the functions $u_h^-,u_h^+$ are given as in \Cref{convofminJ}. We claim that
\[
\Lambda = \bigcup_{j = 2}^{\infty} \bigcup_{n = 1}^{\infty}\Lambda_{j,n}.
\]
Indeed, if $h \in \Lambda$ and $j$ is such that $1/j < h < j$, to prove the claim it is enough to show that $h \in \Lambda_{j,n}$ for some $n$. Assume by contradiction that this is not the case; then it follows from \Cref{monotonicity} that for every $\bm{x} \in B_j$
\[
u_h^-(\bm{x}) = u_h^+(\bm{x}) \ge u_j(\bm{x}) > 0,
\] 
and in turn, reasoning as in the proof of \Cref{monotonicity}, we obtain that $u_h^-$ and $u_h^+$ must coincide in $\Omega$. In view of \Cref{upm}, this contradicts the assumption that $h \in \Lambda$.

Assume that $\Lambda_{j,n}$ has a countable subset. Then we can find a sequence $\{h_i\}_i \subset \Lambda_{j,n}$ and $h \in [1/j, j]$ such that $\{h_i\}_i$ converges strictly monotonically to $h$. By \Cref{convofminJ}, there exists a function $u \in \K_{\ga_h}$ such that $u_{h_i}^- , u_{h_i}^+ \to u$ uniformly in the compact set $\overline{B_j}$. In turn, for $i$ large enough we have that
\[
|u_{h_i}^+(\bm{x}) - u_{h_i}^-(\bm{x})| \le |u_{h_i}^+(\bm{x}) - u(\bm{x})| + |u(\bm{x}) - u_{h_i}^-(\bm{x})| < \frac{1}{n}
\]
for all $\bm{x} \in B_j$, a contradiction with the definition of $\Lambda_{j,n}$. Hence, we have shown that the sets $\Lambda_{j,n}$ are finite for every $j \ge 2$ and $n \in \NN$. This concludes the proof.
\end{proof}

Having established the convergence of monotone sequences of minimizers in \Cref{convofminJ}, we now investigate the convergence of the associated free boundaries. Our proof is inspired by standard techniques which are more commonly used in the study of blow-up limits (see, for example, Section 4.7 in \cite{MR618549}).

\begin{thm}
\label{convFB}
Under the assumptions of \Cref{exofhcrit}, if $\theta$ is continuous at $h > 0$, let $\{h_n\}_n \subset (0,\infty)$ be a monotone sequence that converges to $h$. For every $n \in \NN$, let $u_n$ be a global minimizer of $\J_{h_n}$ in $\K_{\ga_{h_n}}$ and consider $u_h^+,u_h^-$ as in \Cref{upm}. Then the following statements hold:
\begin{itemize}
\item[$(i)$] if $h_n \searrow h$ then $\partial\{u_n > 0\} \to \partial\{u^-_h > 0\}$ in Hausdorff distance locally in $\Omega$;
\item[$(ii)$] if $h_n \nearrow h$ then $\partial\{u_n > 0\} \to \partial\{u^+_h > 0\}$ in Hausdorff distance locally in $\R \times (0,h)$;
\item[$(iii)$] if $h_n \searrow h$  then $\chi_{\{u_n > 0\}} \to \chi_{\{u_h^- > 0\}}$ in $L^1_{\loc}(\R \times (0,h))$;
\item[$(iv)$] if $h_n \nearrow h$  then $\chi_{\{u_n > 0\}} \to \chi_{\{u_h^+ > 0\}}$ in $L^1_{\loc}(\R \times (0,h))$.
\end{itemize}
\end{thm}

\begin{proof} $(i)$ Let $h_n \searrow h > 0$ and consider a ball $B_r(\bm{x}) \subset \Omega$ such that $B_r(\bm{x}) \cap \partial\{u_h^- > 0\} = \emptyset$. Then either $u_h^- \equiv 0$ in $B_r(\bm{x})$ or $u_h^- > 0$ in $B_r(\bm{x})$. By \Cref{monotonicity} we have that for every $n \in \NN$ $\{u_n > 0\} \subset \{u_h^- > 0\}$; thus if $u_h^- \equiv 0$ in $B_r(\bm{x})$ so does $u_n$ for every $n \in \NN$. In particular, this implies that
\begin{equation}
\label{lochausd1}
B_{r/2}(\bm{x}) \cap \partial\{u_n > 0\} = \emptyset.
\end{equation}
On the other hand, if $u_h^- > 0$ in $B_r(\bm{x})$, since by \Cref{convofminJ} we have that $\{u_n\}_n$ converges uniformly to $u_h^-$ in $B_{r/2}(\bm{x})$, then for $n$ sufficiently large
\[
u_n(\bm{x}) \ge \frac{1}{2} \min\left\{u_h^-(\bm{y}) : \bm{y} \in \overline{B_{r/2}(\bm{x})}\right\} > 0
\] 
for every $\bm{x} \in B_{r/2}(\bm{x})$ and hence (\ref{lochausd1}) is satisfied.

Conversely, if $B_r(\bm{x}) \cap \partial\{u_n > 0\} = \emptyset$ then for all $n$ sufficiently large we have that either $u_n > 0$ in $B_r(\bm{x})$ or $u_n = 0$ in $B_r(\bm{x})$. Assume first that $u_{m} > 0$ in $B_r(\bm{x})$ for some $m \in \NN$. Then, by \Cref{monotonicity}, $u_n > 0$ in $B_r(\bm{x})$ for every $n \ge m$ and therefore $u_h^-$ is harmonic in $B_{r/2}(\bm{x})$ being the uniform limit of harmonic functions. Consequently, either $u_h^- > 0$ in $B_{r/2}(\bm{x})$ or $u_h^- = 0$ in $B_{r/2}(\bm{x})$. In both cases 
\begin{equation}
\label{lochausd2}
B_{r/2}(\bm{x}) \cap \partial\{u_h^- > 0\} = \emptyset.
\end{equation} 
On the other hand, if $u_n \equiv 0$ in $B_{r/2}(\bm{x})$ for every $n \in \NN$ then also $u_h^- \equiv 0$ in $B_{r/2}(\bm{x})$. This shows that (\ref{lochausd2}) is also satisfied in case. By a standard compactness argument one can show that $\partial\{u_n > 0\} \to \partial\{u_h^- > 0\}$ in Hausdorff distance locally in $\Omega$.

$(ii)$ Let $h_n \nearrow h$ and consider a ball $B_r(\bm{x}) \subset \R \times (0,h)$ such that $B_r(\bm{x}) \cap \partial\{u_h^+ > 0\} = \emptyset$. As before, either $u_h^+ \equiv 0$ in $B_r(\bm{x})$ or $u_h^+ > 0$ in $B_r(\bm{x})$. If $u_h^+ > 0$ in $B_r(\bm{x})$, by \Cref{monotonicity}, $u_n > 0$ in $B_r(\bm{x})$ for every $n \in \NN$. Therefore (\ref{lochausd2}) holds.
On the other hand, if $u^+ \equiv 0$, for every $\delta > 0$ we can find $m$ such that $u_n \le \delta$ in $B_{3r/4}(\bm{x})$ for every $n \ge m$. Hence, for $\delta = \delta(r)$ sufficiently small and $n \ge m$,
\[
\frac{4}{3r}\fint_{B_{3r/4}(\bm{x})} u_n\,d\mathcal{H}^{N - 1} \le \frac{4\delta}{3r} \le C(N, 2/3)\left(h - x_N - \frac{2}{3}\frac{3}{4}r\right)^b
\]
Then we can conclude from \Cref{competition2} that $u_n \equiv 0$ in $B_{r/2}(\bm{x})$, proving that (\ref{lochausd1}) holds. The rest of the proof follows as in the previous case, therefore we omit the details.

$(iii)$ Let $h_n \searrow h > 0$ and let $K$ be a compact subset of $\R \times (0,h)$. If $\dist(K,\partial \{u_h^- > 0\}) > 0$ then either $u_h^- \equiv 0$ in $K$ or $u_h^- > 0$ in $K$. Reasoning as the proof of $(i)$, we can conclude that either $u_n \equiv 0$ in $K$ for every $n$ or $u_n > 0$ in $K$ for $n$ sufficiently large; hence in this case there is nothing to prove. Therefore, we can assume that $K \cap \partial \{u_h^- > 0\} \neq \emptyset$. By $(i)$, for every $0 < \eta < d_K \coloneqq \dist(K, \partial (\R \times (0,h)))$ we can find $m = m(\eta,K)$ such that if $n \ge m$ then 
\[
\partial \{u_n > 0\}\cap K \subset \mathcal{N}_{\eta}(\partial \{u_h^- > 0\}),
\]
where for any set $A \subset \Omega$, $\mathcal{N}_{\eta}(A)$ represents the tubular neighborhood of $A$ of width $\eta$, i.e.,
\[
\mathcal{N}_{\eta}(A) \coloneqq \{\bm{x} \in \Omega : \dist(\bm{x}, A) < \eta\}.
\]
Observe that by \Cref{competition2}, for every ball $B_r(\bm{x}) \subset K$ with center on $\partial \{u_h^- > 0\}$ 
\[
\frac{1}{r}\fint_{\partial B_r(\bm{x})}u_h^-\,d\mathcal{H}^{N - 1} \ge C(N, 1/2)(h - x_N - r/2)^b > C(N, 1/2)(d_K)^b.
\]
Similarly, by Lemma 3.2 in \cite{MR618549} (see also Theorem 3.1 in \cite{MR2915865}), there is a constant $C_{\max} = C_{\max}(N) > 0$ such that 
\[
\frac{1}{r}\fint_{\partial B_r(\bm{x})}u_h^-\,d\mathcal{H}^{N - 1} \le C_{\max}(h - x_N + r)^b < C_{\max}(2h)^b.
\]
Hence we are in a position to apply Theorem 4.5 in \cite{MR618549} to conclude that 
\[
\mathcal{H}^{N - 1}(\partial \{u_h^- > 0\} \cap K) < \infty.
\]
Since $\chi_{\{u_n > 0\}} \to \chi_{\{u_h^- > 0\}}$ in $L^1(K \setminus \mathcal{N}_{\eta}(\partial \{u_h^- > 0\}))$ and since
\[
\mathcal{L}^N(\mathcal{N}_{\eta}(\partial \{u_h^- > 0\}) \cap K) \le 2 \eta \mathcal{H}^{N - 1}(\partial \{u_h^- > 0\} \cap K),
\]
letting $\eta \to 0^+$ in the previous estimate concludes the proof.

The proof of $(iv)$ is almost identical, thus we omit the details. 
\end{proof}

The following result is adapted from Theorem 5.10 in \cite{MR2915865}.
\begin{thm}
\label{symthm}
Let $u_h^+, u_h^-$ be as in \Cref{upm}. Then $u_h^+, u_h^-$ are symmetric with respect to the coordinate hyperplanes $\{x_i = 0\}$ and the maps 
\[
x_i \in [0,\la/2] \mapsto u_h^+(\bm{x}), \quad x_i \in [0,\la/2] \mapsto u_h^-(\bm{x})
\] 
are decreasing for $i = 1, \dots, N-1$.
\end{thm}
\begin{proof}
\textbf{Step 1:} Let $h \in \RR^+ \setminus \Lambda$ where $\Lambda$ is defined as in \Cref{unicitamin} and let $u_h$ be the unique global minimizer of $\J_h$ in  $\K_{\ga_h}$. For $i = 1, \dots , N - 1$, let $w_i$ be the function obtain by applying to $u_h$ an even reflection about the hyperplane $\{x_i = 0\}$, i.e.
\[
w_i(\bm{x}) \coloneqq 
\left\{
\arraycolsep=1.4pt\def\arraystretch{1.6}
\begin{array}{ll}
u_h(-x_1,x_2,\dots,x_N) & \text{ if } i = 1, \\
u_h(x_1,\dots,-x_i,\dots,x_N) & \text{ if } i \ge 2.
\end{array}
\right.
\]
Notice that $w_i \in \K_{\ga_h}$ and $\J_h(w) = \J_h(u_h)$. Thus, since by assumption $\J_h$ has exactly one global minimizer in $\K_{\ga_h}$, it must be the case that $u_h = w_i$ for every $i$. This proves that $u_h$ is symmetric with respect to the hyperplanes $\{x_i = 0\}$ for $i = 1, \dots, N - 1$, and in particular the support of $u_h$ in $\Omega$ coincides with its Steiner symmetrizations with the respect to the same hyperplanes. Let $u_h^*$ be the symmetric decreasing rearrangement of $u_h$ with respect to the variables $x_1, \dots, x_{N - 1}$ (see Chapter 2 in \cite{MR810619}; see also Definition 7.1 in \cite{MR1009785}). Then $u_h^* \in \K_{\ga_h}$ and by the P\'olya-Szeg\"o inequality (see Corollary 2.14 \cite{MR810619}; see also Theorem 7.1 in \cite{MR1009785}), together with Tonelli's theorem and Lebesgue's monotone convergence theorem, we obtain
\[
\int_{\Omega}|\nabla u_h^*|^2\,d\bm{x} \le \int_{\Omega}|\nabla u_h|^2\,d\bm{x}.
\]
Furthermore, the definition of $u_h^*$ implies that for $\mathcal{L}^1$-a.e.\@ $x_N \in \RR_+$ 
\[
\int_{\R}\chi_{\{u_h^* > 0\}}(\bm{x}',x_N)\,d\bm{x}' = \int_{\R}\chi_{\{u_h > 0\}}(\bm{x}',x_N)\,d\bm{x}', 
\]
and thus, again by Tonelli's theorem,
\begin{align*}
\int_{\Omega}\chi_{\{u_h^* > 0\}}(h - x_N)_+^{2b}\,d\bm{x} = &\ \int_0^h(h - x_N)_+^{2b}\int_{\R}\chi_{\{u_h^* > 0\}}(\bm{x}',x_n)\,d\bm{x}'\,dx_N \\
= &\ \int_0^h(h - x_N)_+^{2b}\int_{\R}\chi_{\{u_h > 0\}}(\bm{x}',x_n)\,d\bm{x}'\,dx_N \\
= &\ \int_{\Omega}\chi_{\{u_h > 0\}}(h - x_N)_+^{2b}\,d\bm{x}.
\end{align*}
Consequently, $\J_h(u_h^*) \le \J_h(u_h)$, which in turn gives that $u_h \equiv u_h^*$.
\newline
\textbf{Step 2:} If $h \in \Lambda$, consider a sequence $\{h_n\}_n \subset \RR \setminus \Lambda$ such that $h_n \nearrow h$ and let $u_{h_n}$ be the unique minimizer of $\J_{h_n}$ in $\K_{\ga_{h_n}}$. Then, $u_{h_n} \equiv u_{h_n}^*$ and by \Cref{convofminJ} it follows that $u_h^+$ has all the desired properties. The analogous result for $u_h^-$ follows by considering a sequence $\{h_n\}_n \subset \RR \setminus \Lambda$ such that $h_n \searrow h$.
\end{proof}

\begin{rmk}
\label{graph}
Let $u_h \in \K_{\ga_h}$ be a global minimizer of $\J_h$, assume that the map
\[
x_i \in [0,\la/2] \mapsto u_h(\bm{x})
\] 
is decreasing for some $i \in \{1, \dots, N-1\}$. Then the free boundary of $u_h$ in $(0,\la/2)^{N - 1} \times \RR_+$ can be described by the graph of a function 
\[
x_i = g_i(\hat{\bm{x}}_i),
\]
where the vector $\hat{\bm{x}}_i$ is obtained from $\bm{x}$ by removing the entry corresponding to $x_i$. Indeed, it is enough to define 
\[
g_i(\hat{\bm{x}}_i) \coloneqq \sup\{x_i : u_h(\bm{x}) > 0\}.
\]
\end{rmk}

Notice that \Cref{structure} follows directly from \Cref{convofminJ}, \Cref{upm}, \Cref{unicitamin}, and \Cref{symthm}.

\section{Comments}
It is important to observe that \Cref{exofhcrit} implies that the critical height $h_{\cri}$ is the only value of $h$ for which the free boundaries of global minimizers of $\J_h$ can touch the hyperplane $\{x_N = h\}$ while having support contained in $\{x_N \le h\}$. As previously observed in \cite{MR2915865} in dimension $N = 2$, it follows from \Cref{convofminJ} and \Cref{competition2} that the support of $u_{h_{\cri}}^+$ cannot be strictly contained in $\{y < h_{\cri}\}$, while the support of $u_{h_{\cri}}^-$ cannot cross the line $\{y = h_{\cri}\}$. In turn, a necessary condition for the existence of a minimizer with the desired properties is that $u_{h_{\cri}}^+ \equiv u_{h_{\cri}}^-$. As previously remarked in the introduction, our interest in the matter is due to the fact that in view of the results of Theorem 5.11 in \cite{MR2915865}, such a minimizer would behave as a Stokes wave locally in $\Omega$. 

\section*{Acknowledgements}
This paper is part of the first author's Ph.\@ D.\@ thesis at Carnegie Mellon University. The authors acknowledge the Center for Nonlinear Analysis (NSF PIRE Grant No.\@ OISE-0967140) where part of this work was carried out. The research of G.\@ Gravina and G.\@ Leoni was partially funded by the National Science Foundation under Grants No.\@ DMS-1412095 and DMS-1714098. G.\@ Gravina also acknowledges the support of the research support programs of Charles University: PRIMUS/19/SCI/01 and UNCE/SCI/023. G.\@ Leoni would like to thank Ovidiu Savin and Eugen Varvaruca for their helpful insights. The authors would also like to thank Luis Caffarelli, Ming Chen, Craig Evans and Ian Tice for useful conversations on the subject of this paper.

\bibliographystyle{alpha}
\bibliography{waterwaves}
\addcontentsline{toc}{section}{References}

\end{document}